\documentclass[12pt,reqno]{amsart}

\usepackage{amsmath,amsthm,amscd,amsfonts,amssymb,epic,eepic,bbm,dsfont,bm}
\usepackage[colorlinks=true,linkcolor=blue,citecolor=blue]{hyperref}

\usepackage{tikz}
\usetikzlibrary{arrows,matrix}

\newcommand{\opm}{ 
  \mathbin{
    \mathchoice
      {\buildcirclepm{\displaystyle     }{0.14ex}{0.95}{0.05ex}{.7}}
      {\buildcirclepm{\textstyle        }{0.14ex}{0.95}{0.05ex}{.7}}
      {\buildcirclepm{\scriptstyle      }{0.13ex}{0.955}{0.04ex}{.55}}
      {\buildcirclepm{\scriptscriptstyle}{0.08ex}{0.95}{0.03ex}{.45}}
  } 
}

\newcommand\buildcirclepm[5]{%
  \begin{tikzpicture}[baseline=(X.base), inner sep=-#5, outer sep=-.65]
    \node[draw,circle,line width=#4] (X)  {\footnotesize\raisebox{#2}{\scalebox{#3}{$#1\pm$}}};
  \end{tikzpicture}%
}

\allowdisplaybreaks
\newtheorem*{main}{The main theorem}
\newtheorem{thm}{Theorem}[section]

\newtheorem{lem}[thm]{Lemma}

\newtheorem{prop}[thm]{Proposition}
\theoremstyle{remark}
\newtheorem{rem}{Remark}

\newtheorem{innercustomgeneric}{\customgenericname}
\providecommand{\customgenericname}{}
\newcommand{\newcustomtheorem}[2]{%
  \newenvironment{#1}[1]
  {%
   \renewcommand\customgenericname{#2}%
   \renewcommand\theinnercustomgeneric{##1}%
   \innercustomgeneric
  }
  {\endinnercustomgeneric}
}

\newcustomtheorem{customcor}{\normalfont{\textbf{Corollary}}}
\newcustomtheorem{customlem}{\normalfont{\textbf{Lemma}}}
\newcustomtheorem{customprop}{\normalfont{\textbf{Proposition}}}

\numberwithin{equation}{section}

\newcommand{\quash}[1]{}

\theoremstyle{definition}

\numberwithin{equation}{subsection}

\title{Secondary terms in asymptotics for the number of zeros of quadratic forms}
\author{Thomas H. Tran
        }
\thanks{
                  Department of Mathematics, Duke University, tran2015@math.duke.edu}
\begin{document}
\newpage
\maketitle
\begin{abstract}

Let $F$ be a non-degenerate quadratic form on an $n$-dimensional vector space $V$ over the rational numbers. One is interested in counting the number of zeros of the quadratic form whose coordinates are restricted in a smoothed box of size $B$, roughly speaking. For example, Heath-Brown gave an asymptotic of the form: $c_1 B^{n-2} +O_{J,\epsilon, \omega}(B^{(n-1)/2+\epsilon})$, for any $\epsilon > 0$ and dim$V \geq 5$, where $c_1 \in \mathbb{C}$ and $\omega \in \mathcal{S}(V(\mathbb{R}))$ is a smooth function. More recently, Getz gave an asymptotic of the form: $c_1 B^{n-2} + c_2 B^{n/2}+O_{J,\epsilon, \omega}(B^{n/2+\epsilon-1})$ when $n$ is even, in which $c_2 \in \mathds{C}$ has a pleasant geometric interpretation. We consider the case where $n$ is odd and give an analogous asymptotic of the form: $c_1 B^{n-2} +c_2B^{(n-1)/2}+O_{J,\epsilon,\omega}(B^{n/2+\epsilon-1})$. Notably it turns out that the geometric interpretation of the constant $c_2$ of the asymptotic in the odd degree and even degree cases is strikingly different.
\end{abstract}

\tableofcontents

\section{Introduction and Statement of Results}

The circle method, introduced by G. H. Hardy and S. Ramanujan in 1918 and then improved by G. H. Hardy and J. E. Littlewood a few years later, is usually applied to count the number of zeros of Diophantine equations. In most cases, one only obtains the main term in the asymptotic. This changed in the recent papers of Schindler, Getz, Vaughn and Wooley (see \cite{S,G,VW}). In this paper, we continue this investigation on quadratic form following the Health-Brown's new version of the circle method, which is rooted from a double-summation representation of the $\delta$-symbol that we will address in the next section. We consider odd number of variables of the quadratic form and wish to count the number of its zeros whose heights are at most $B$, where real number $B \geq 1$. Roughly, we obtain an asymptotic of the form $$c_1 B^{n-2} + c_2 B^{(n-1)/2}+O_{J,\epsilon, \omega}(B^{n/2+\epsilon-1}),$$ where odd $n \geq 5$ is the number of variables of the quadratic form. In the boundary case $n=3$, our asymptotic has the form  
\begin{align*}
   c_1' B \log B + c_2' B +O_{J,\epsilon, \omega}(B^{1/2+\epsilon}).
\end{align*}
These asymptotic formulas confirm the sharpness of Heath-Brown's estimate (see \cite{H-B},Theorems 5 and 8).   

To state our main theorem, we briefly introduce our objective counting function, $N(B)$, which is defined by \begin{align*}
    N(B):= N(B,\omega,J)= \sum_{x \in \mathbb{Z}^n} \omega\left(\frac{x}{B}\right)\delta_{F(x)},
    \end{align*}
where $J$ is the symmetric matrix associated to the quadratic form $F(x)$, and $\omega(x)$ is a smooth compactly supported function in $\mathbb{R}^n$.  We assume for simplicity that  $J$ is diagonal with entries $b_i=\pm 1$, where $i=1, \dots,n$, though our methods are completely general. 
In the definition of $N(B)$, the $\delta$-symbol plays a role in detecting zeros of $F(x)$, while $\omega(x)$, which can be think of a bump function, assures the zeros are restricted in a smoothed box of size $B$. 

We adapt the method of Heath-Brown to initiate and transform the zero-counting function $N(B)$ to a more familiar representation. Indeed, $N(B)$ can be written in terms of a sum of contour integrals: 
\begin{align*}
\sum_{c \in \mathbb{Z}^n} \int_{\mathrm{Re}(s)=\sigma_0}D(c,s)B^{s}g(c,s)\,ds, \tag{1.1} \label{eq1.1}
\end{align*}
where $\sigma_0$ is chosen appropriately so that the Dirichlet series $D(c,s)$ converges absolutely, and $g(c,s)$ is defined in terms of Fourier transform and Mellin transform as in (\ref{eq2.6}).

It turns out that the analytic continuation of $D(c,s)\,g(c,s)$ ultimately determines the shape of the asymptotic of $N(B)$. In other words, the further one can continue $D(c,s)$ analytically, viewed a function of $s$, the more accurate the asymptotic of interest is. In fact, Heath-Brown gave a meromorphic continuation of  $D(c,s)$ to the half plane $\mbox {Re}(s) > 1-n/2-\alpha$ for some $\alpha > 0$, which is good enough to obtain the asymptotics proven in \cite{H-B}. In order to acquire the secondary term, in \cite{G}, Getz extended the regime of meromorphic continuation of $D(c,s)$ to $\mbox {Re}(s) > -n/2$. Getz's arguments use the fact that $n$ is even, and it is not obvious how to extend the theorem proven in \cite{G} to the odd degree case. Motivating by the works of Health-Brown and Getz, I work classically to exploit Gauss character sums, find the Dirichlet characters that fit into the odd degree case, and also give an explicit expression of $D(c,s)$, at least outside the archimedian place, which is continued meromorphically to the regime of interest. Here is our main theorem:
\begin{main}
\label{thm1}
Let $\epsilon > 0$ and denote by $\square$ a perfect square. If $n \geq 5$ is odd,
\begin{align*}
    N(B)=c_1 B^{n-2} + c_2 B^{(n-1)/2}+O_{J,\epsilon, \omega}(B^{n/2+\epsilon-1}), \tag{1.2} \label{eq1.2}
\end{align*}
where
\begin{align*}
    c_1&=\mbox{\normalfont{Res}}_{\,s=-1}\,(D(0,s)\,g(0,s)),\\
    \mbox{ and } 
    c_2&=
   \sum_{\substack{c \, \in \, \mathds{Z}^n\\ c^t J^{-1} c \,=\,0}} \mbox{\normalfont{Res}}_{\,s=1/2-n/2}\,(D(c,s)\,g(c,s))+ \\& \quad + \sum_{\substack{0 \, \neq \, \, c \, \in \, \mathds{Z}^n\\ c^t J^{-1} c \, \neq \, 0 \\ i^{(n-1)^2/2}\,c^t J^{-1} c\,\det J\,  =\, \square}} \mbox{\normalfont{Res}}_{\,s=1/2-n/2}\,(D(c,s)\,g(c,s)). \tag{1.3} \label{eq1.3}
    \end{align*}
Moreover, if $n=3$,
\begin{align*}
    N(B)=c_1' B \log B + c_2' B +O_{J,\epsilon, \omega}(B^{1/2+\epsilon}), \tag{1.4} \label{eq1.4}
\end{align*}
where 
\begin{align*}
    c_1'&=\frac{-4}{\pi^4} \int_{\mathbb{R}} Z( z,1) \, dz, \tag{1.5} \label{eq1.5} \\
    \mbox{ and } 
    c_2'&=
    \sum_{\substack{c \, \in \, \mathbb{Z}^n\\ c^t J^{-1} c\, =\,0}} \mbox{\normalfont{Res}}_{\,s=-1}\,(D(c,s)\,g(c,s))+\sum_{\substack{0 \, \neq \, \, c \, \in \, \mathbb{Z}^n\\ c^t J^{-1} c \, \neq \, 0 \\ -c^t J^{-1} c\,\det J\,  =\, \square}} \mbox{\normalfont{Res}}_{\,s=-1}\,(D(c,s)\,g(c,s)). \tag{1.6} \label{eq1.6}
\end{align*}
\end{main}
\begin{rem}
Before we dive into our investigation, we want to make some comments here. 
\begin{itemize}
    \item The function $Z( z,1)$ in (\ref{eq1.5}) is defined as (\ref{eq5.1}) in Section \ref{Sec 5}.
    \item Our asymptotic formulas (\ref{eq1.2}) and (\ref{eq1.4}) are fundamentally consistent with those of Getz's paper \cite{G}. However, it is notable that the secondary terms (\ref{eq1.3}) and (\ref{eq1.6}) involve the solutions to a quadratic form in $n+1$ variables, $i^{(n-1)^2/2}\,c^t J^{-1} c\,\det J  = \square$, not just a quadratic form in $n$ variables, $c^t J^{-1} c =0$.  This is in marked contrast to the corresponding results in \cite{G}.
    \item In \cite{G}, the secondary term is given in a less explicit manner.  In this paper, we work classically and give an explicit description of the secondary term, at least outside the infinite place, which sheds light on the questions of whether the secondary term is identically to zero and whether it is positive or negative in the odd degree case of quadratic forms. This information will be useful in applying the theorem. 
    \item The $D(c,s)\,g(c,s)$, viewed as a function of $s$, admits the meromorphic continuation to $\mbox{Re}(s) > -n/2$ with possible simple poles at $s=0, -1,$ and $ 1/2-n/2.$ For application, it may be necessary to know the existence conditions of these poles. We presented these results in Remark \ref{Remark} and Appendix \ref{appendix:Appendix}.
\end{itemize}
\end{rem}
We postpone the proof of our main theorem until the last section, Section \ref{Sec 6}, of this paper. We address the double-summation expression of $\delta$-symbol and formulate $N(B)$ to its contour integral representation (\ref{eq1.1}) in Section \ref{Sec 2}. We then study exponential double summations, namely $S(q,c)$ as defined in (\ref{eq2.3}), which are the coefficients of the Dirichlet series $D(c,s)$, and subsequently analyze $D(c,s)$ in Sections \ref{Sec 3} and \ref{Sec 4}. These two sections are the most technical parts of the paper. Using Getz's results, we briefly treat $g(c,s)$ in Section \ref{Sec 5}. We entirely work and use classically analytic language throughout the paper.


\section{Preliminary steps}
\label{Sec 2}
We start our investigation by adapting a double-summation expression of the $\delta$-symbol, which was essentially due to Duke, Friedlander, and Iwaniec (see \cite{DFI}, Section 2), and then later by Health-Brown (see \cite{H-B}, Theorem 1). In contrast to Health-Brown, we keep the $\delta$-symbol in terms of non-Ramanujan sum.
To ease notation, we denote $e(x)$ by $e^{2 \pi i x}$, and $c_{p,j,J}:=(c/p^j)^t J^{-1} (c/p^j)$, where $t$ means transpose. We use the subscript $L$ to denote both the Legendre symbol and its extension, the Jacobi symbol, while the subscript $K$ stands for the Kronecker symbol. Also, the notation $\, \mathds{1} _S$ is $1$ if the statement $S$ is true and $0$ otherwise.


\begin{prop}
\label{Prop 2.1}
Assume $\Phi(x,y) \in \mathcal{S}(\mathbb{R}^2)$ with $\Phi(x,0)=0$ for all $x$ and $\int_{\mathbb{R}}\Phi(0,y) \, dy=1$. For any integer $n$ and any real number $Q > 1$, there is a positive constant $c_Q$ such that
\begin{align}
    \delta_n=c_Q Q^{-1} \sum_{q=1}^{\infty} \frac{1}{q} \sum_{a (\mbox{\tiny{mod} }  q)} e\left(\frac{a n}{q}\right) h\left( \frac{n}{q\,Q}, \frac{q}{Q}\right), \tag{2.1} \label{eq2.1}
\end{align}
where $c_Q=1+O_N(Q^{-N})$ for any positive integer $N$, and 
$\displaystyle{h(x,y)= \Phi(x,y) - \Phi(y,x)}$. 
\end{prop}
\begin{proof}
Since $\frac{n}{q}$ runs over all divisors of $n$ as $q$ does, we have
\begin{align*}
                \sum_{q|n} \left[\Phi\left(\frac{n}{q\, Q},\frac{q}{Q}\right)-\Phi\left(\frac{q}{Q},\frac{n}{q\,Q}\right)\right]=      \sum_{q \in \mathbb{Z}}\Phi\left(0, \frac{q}{Q} \right) \, \mathds{1} _{n=0} .           \end{align*}
Since $\Phi$ is Schwartz, Poisson summation yields
\begin{align*}
  \sum_{q \in \mathbb{Z}}\Phi\left(0, \frac{q}{Q} \right) &=  \sum_{q \in \mathbb{Z}} \int_{\mathbb{R}} \Phi\left(0, \frac{x}{Q}  \right) e(-xq) \, dx.
\end{align*}
By integrating by parts and changing variables, the inner integral is equal to $Q$ when $q=0$ and $O_N(Q (Q|q|)^{-N})$ for all $q \, \neq \, 0.$
Then, we set $c_Q:=1+O_N(Q^{-N})$ and get
\begin{align*}
  \sum_{q \in \mathbb{Z}}\Phi\left(0, \frac{q}{Q} \right) &=  c_Q^{-1}Q,
\end{align*}
which implies
\begin{align*}
    \delta_n= c_Q Q^{-1}\sum_{q|n} \left[\Phi\left(\frac{n}{q\, Q},\frac{q}{Q}\right)-\Phi\left(\frac{q}{Q},\frac{n}{q\,Q}\right)\right] =  c_Q Q^{-1}\sum_{q|n} h\left( \frac{n}{q\,Q}, \frac{q}{Q}\right),
\end{align*}
where $\displaystyle{h(x,y)= \Phi(x,y) - \Phi(y,x)}$.
Thus, the proposition follows by applying orthogonality of additive characters (see \cite{IK}, Chapter 3) to detect divisors $q$ of $n$. 
\end{proof}
Our smoothed counting function is defined by $$N(B)= \sum_{x \in \mathbb{Z}^n} \omega\left(\frac{x}{B}\right)\delta_{F(x)}.$$
We apply Proposition \ref{Prop 2.1} and write $x=x_0+qy$ to get
\begin{align*}
N(B)
&= c_Q Q^{-1} \sum_{q=1}^{\infty} \frac{1}{q} \sum_{a (\mbox{\tiny{mod} }  q)} \sum_{x \in \mathbb{Z}^n} \omega\left(\frac{x}{B}\right) e\left(\frac{a F(x)}{q}\right) h\left( \frac{F(x)}{q\,Q}, \frac{q}{Q}\right)\\
&= c_Q  Q^{-1}  \sum_{q=1}^{\infty} \frac{1}{q} \sum_{a (\mbox{\tiny{mod} }  q)}\sum_{x_0 \in {(\mathbb{Z}/q\mathbb{Z})}^n}  e\left(\frac{a F(x_0)}{q}\right)\sum_{y \in \mathbb{Z}^n}\omega\left(\frac{x_0+qy}{B}\right) h\left( \frac{F(x_0+qy)}{q\,Q}, \frac{q}{Q}\right). 
\end{align*}
Then, the $n$-dimensional Poisson summation applied to $\omega\left(\frac{x_0+qy}{B}\right) h\left( \frac{F(x_0+qy)}{q\,Q}, \frac{q}{Q}\right)$, which is Schwartz as a function of $y$, yields
\begin{align*}
N(B)
&= c_Q  Q^{-1} \sum_{c \in \mathbb{Z}^n} \sum_{q=1}^{\infty}q^{-n-1} \, S(q,c) \,\int_{\mathbb{R}^n} \omega\left(\frac{x}{B}\right) h\left( \frac{F(x)}{q\,Q}, \frac{q}{Q}\right) e\left(-\frac{c\cdot x}{q}\right) \,dx, \tag{2.2} \label{eq2.2}
\end{align*}
where \begin{align*}
    S(q,c)=\sum_{a (\mbox{\tiny{mod} }  q)} \sum_{x_0 \in {(\mathbb{Z}/q\mathbb{Z})}^n}  e\left(\frac{a F(x_0)+c\cdot x_0}{q}\right). \tag{2.3} \label{eq2.3}
\end{align*}
Here, we interchange summations in (\ref{eq2.2}). It is permissible in the Schwartz space in which our test functions lie.
We end this section by representing $N(B)$ as a sum of contour integrals:
\begin{prop}
\label{Prop 2.2}
\begin{align*}
N(B)&= \frac{c_B  B^{n-1}}{2 \pi i} \sum_{c \in \mathbb{Z}^n} \int_{\mathrm{Re}(s)=\sigma_0}D(c,s)B^{s}g(c,s)\,ds, \tag{2.4} \label{eq2.4}
\end{align*}
where 
\begin{align*}
    D(c,s)=\sum_{q=1}^{\infty} \frac{S(q,c)}{q^{s+n+1}}, \tag{2.5} \label{eq2.5}
    \end{align*}
    \begin{align*}
 g(c,s)=  \int_{0}^\infty \int_{\mathbb{R}^n} \omega(x)\,h\left(\frac{F(x)}{y},y\right)e\left(-\frac{c\cdot x}{y}\right)  \,dx \, y^s \, \frac{dy}{y}, \tag{2.6} \label{eq2.6}
\end{align*}
with  $S(q,c)$ is defined in (\ref{eq2.3}), and $\sigma_0$ is chosen large enough so that the Dirichlet series $D(c,s)$ converges absolutely.
\end{prop}
\begin{proof} By using trivial bound $q^{n+1}$ of $S(q,c)$, it is clear that the Dirichlet series $D(c,s)$ converges absolutely as long as  $\sigma_0 > 1$. Together with Mellin inversion, we then have 
\begin{align*}
    &\frac{1}{2\pi i}\int_{\mathrm{Re}(s)=\sigma}D(c,s)\,\int_{0}^\infty \int_{\mathbb{R}^n} \omega\left(\frac{x}{B}\right) h\left( \frac{F(x)}{yQ}, \frac{y}{Q}\right) e\left(-\frac{c\cdot x}{y}\right) \,dx\,y^s\frac{dy}{y} \,ds \tag{2.7} \label{eq2.7} \\
    &= \sum_{q=1}^\infty q^{-n-1}   S(q,c)  \frac{1}{2\pi i}\int_{\mathrm{Re}(s)=\sigma} \,\int_{0}^\infty \int_{\mathbb{R}^n} \omega\left(\frac{x}{B}\right) h\left( \frac{F(x)}{yQ}, \frac{y}{Q}\right) e\left(-\frac{c\cdot x}{y}\right) \,dx\,y^s\frac{dy}{y} \, q^{-s} \,ds \\
    &= \sum_{q=1}^{\infty}q^{-n-1} \, S(q,c) \,\int_{\mathbb{R}^n} \omega\left(\frac{x}{B}\right) h\left( \frac{F(x)}{q\,Q}, \frac{q}{Q}\right) e\left(-\frac{c\cdot x}{q}\right) \,dx,
    \end{align*}
which is the inner sum of $N(B)$ in (\ref{eq2.2}).
We also make our choice of $Q=B$. By changing variables $(x,y) \mapsto (Bx,By)$ and utilizing homogeneous property of quadratic form, we can extract $B$ out of the two inner integrals in (\ref{eq2.7}) and obtain
\begin{align*}
    B^{s+n} \,\int_{0}^\infty \int_{\mathbb{R}^n} \omega(x)\,h\left(\frac{F(x)}{y},y\right)e\left(-\frac{c\cdot x}{y}\right)  \,dx \, y^s \, \frac{dy}{y}.
\end{align*}
Hence, we have proved the proposition.
\end{proof}

\section{Exponential sum}
\label{Sec 3}
In order to analyze the Dirichlet series $D(c,s)$ (\ref{eq2.5}) as defined in Section \ref{Sec 2}, we devote this section on studying its coefficients $S(q,c)$, which are exponential double sums given by
$$S(q,c)=\sum_{a (\mbox{\tiny{mod} }  q)} \sum_{x_0 \in {(\mathbb{Z}/q\mathbb{Z})}^n}  e\left(\frac{a F(x_0)+c\cdot x_0}{q}\right)  . $$
By the mean of Chinese remainder theorem, $S(q,c)$ is  multiplicative as a function of $q$ (see also \cite{H-B}, page 45).
In other words,
\begin{align*}
    S(q,c)=\prod_{p^k\,||\,q} S(p^{k},c)
    &=\prod_{p^k\,||\,q} \sum_{a (\mbox{\tiny{mod} }  p^{k})} \sum_{x_0 \in {(\mathbb{Z}/p^{k}\mathbb{Z})}^n}  e\left(\frac{a F(x_0)+c\cdot x_0}{p^{k}}\right) \\
    &= \prod_{p^k\,||\,q} \sum_{a (\mbox{\tiny{mod} }  p^{k})} p^{n k/2} \,G_{k}(2aJ,c) , \tag{3.1} \label{eq31}
  \end{align*}
  where $G_k(A,v)=p^{-nk/2}\sum_{x \in (\mathbb{Z}/p^k\mathbb{Z})^n} e^{( \pi i x^t A x/p^k)}e^{( 2 \pi i v \cdot x/p^k )}$ is an $n$-dimensional Gauss sum associated with a $n \times n$ symmetric matrix $A$, denoted by $A \in S_{n \times n} $, and $n$-dimensional vector $v$, and also write $G_k(A):=G_k(A,0).$
  
In \cite{DF}, Dabrowski and Fisher state properties of $G_k(A,v)$ in Lemma \ref{Lem 3.1}, which can be verified by elementary linear algebra. 
\begin{lem}
\label{Lem 3.1}
The $n$-dimensional Gauss sum $G_k(A,v)$, where $A \in S_{n \times n}(\mathbb{Z}_p)$ and $v \in \mathbb{Z}_p^n$, satisfies
\begin{align*}
   \qquad \qquad \qquad G_k(A,v)&=G_k(P^tAP,P^tv), \mbox { where } P \in \mbox{GL}_n(\mathbb{Z}_p), \\
    \mbox { and } G_k(A)&=G_k(A_1)\,G_k(A_2) \mbox { if } A=
\begin{bmatrix}
A_1 &  \\
 & A_2
\end{bmatrix}
 \mbox { is block-diagonal.}  \hspace{.85 in} \square
\end{align*}
\end{lem}
In fact, the second property in Lemma \ref{Lem 3.1} also holds for $G_k(A,v)$, i.e., $G_k(A,v)=G_k(A_1,v_1)\,G_k(A_2,v_2)$. Moreover, Dabrowski and Fisher give a proof of Proposition \ref{Prop 3.2}, but we were unable to follow their argument. We give an alternative proof for the convenience of the reader.
\begin{prop}
\label{Prop 3.2}
Let $A \in S_{n \times n}(\mathbb{Z}_p)$ and $v \in \mathbb{Z}_p^n$. For odd prime $p$,
$$G_k(A,v) \, \neq \, 0 \mbox { if and only if } v=Au \mbox { for some } u \in \mathbb{Z}_p^n.$$ 
Moreover, if $v=Au$, then $\displaystyle{
    G_k(A,v)
    = e^{( - \pi i u^t A u/p^k )}G_k(A) }$.
\end{prop}
\begin{proof}
By replacing $x$ by $x+y$, where $ y \in (\mathbb{Z}/p^k\mathbb{Z})^n$, and noting that 
\begin{align*}
    (x+y)^t A (x+y) =x^tAx+2Ay\cdot x+y^tAy,
\end{align*}
we get $G_k(A,v)
        = e^{( \pi i y^t A y/p^k)}e^{( 2 \pi i v \cdot y/p^k )}G_k(A,v+Ay)$.
If $v=Au$, we choose $y=-u$ and obtain
\begin{align*}
    G_k(A,v)&= e^{( \pi i u^t A u/p^k)}e^{( -2 \pi i u^t A u/p^k )}\,G_k(A) \\
    &= e^{( - \pi i u^t A u/p^k )}\,G_k(A) \, \neq \, 0. 
  \end{align*} On the other hand, as prime $p$ is odd, there exists $P \in \mbox{GL}_n(\mathbb{Z}_p) $ such that $P^t A P$ is diagonal (see \cite{DF}, Remark 1.14) with entries $u_1p^{a_1}, u_2p^{a_2}, \ldots, u_np^{a_n}$, where $p \nmid u_i$ for all $i=1, \ldots, n$. By Lemma \ref{Lem 3.1} and the fact that $v \in \mathbb{Z}_p^n$, it suffices to consider the one-dimensional Gauss sum and to show that for each $i$, $$\,G_k(u_i p^{a_i},v_i)
=0 \mbox { whenever } p^{a_i} \nmid v_i.$$ 
Indeed, the sum is invariant by changing $x_i$ to $x_i+yp^{k-a_i}$, for some non-zero $y \in \mathbb{Z}/p^{a_i}\mathbb{Z}$. That is,
\begin{align*}
    G_k(u_ip^{a_i},v_i) &= p^{-nk/2}\sum_{x_i \in \mathbb{Z}/p^k\mathbb{Z}} e^{( \pi i x_i^2 u_i p^{a_i} /p^k)} e^{( \pi i y^2 u_i p^{k-a_i})}e^{( 2 \pi i v_i x_i/p^k )} e^{(2 \pi i v_i p^{-a_i} y)} \\
    &= e^{ \pi i y^2 u_i p^{k-a_i}} e^{2 \pi i v_i u_i p^{-a_i} y} \,G_k(u_i p^{a_i},v_i),
\end{align*}
which implies $G_k(u_i p^{a_i},v_i)=0$ whenever $p^{a_i} \nmid v_i$.
\end{proof}
In view of Proposition \ref{Prop 3.2}, the multiplicative property of $S(q,c)$ regarded as a function of $q$, and the sake of simplicity, we now assume the symmetric matrix $J$ associated to the quadratic form $F(x)$ is diagonal with entries $b_i=\pm1$, where $i=1, \ldots, n$, which implies that Proposition \ref{Prop 3.2} applied to $G_k(2aJ,c)$ in (\ref{eq31}) is valid for all primes and $J^{-1}=J$. As discussed in Section \ref{Sec 2}, the expansion of the $\delta$-symbol (\ref{eq2.1}) we use is not in the form of Ramanujan sum with respect to $a$. To facilitate our upcoming computation and to take quadratic Gauss sums into account, we turn $S(p^{k},c)$ into a sum of Ramanujan-type sums in Proposition \ref{Prop 3.3}:
\begin{prop}
   \label{Prop 3.3}
   \begin{align*}
       S(p^{k},c)= \sum_{j=0}^{k}\sum_{a (\mbox{\tiny{mod }} p^{k-j})}^* \sum_{x_0 \in {(\mathbb{Z}/p^{k}\mathbb{Z})}^n}  e\left(\frac{a F(x_0)+\frac{c}{p^j}\cdot x_0}{p^{k-j}} \right) \, \mathds{1} _{p^j|c},
   \end{align*}
   where the $*$ above the summation indicates $(a,p^{k-j})=1$.
\end{prop}
\begin{proof}
   We observe 
   \begin{align*}
       \mathbb{Z}/p^{k} = (\mathbb{Z}/p^{k})^{\times} \cup p(\mathbb{Z}/p^{k-1})^{\times} \cup \ldots \cup p^j(\mathbb{Z}/p^{k-j})^{\times} \cup \ldots \cup p^{k-1}(\mathbb{Z}/p)^{\times}\cup 0,
   \end{align*}
     and then rewrite $S(p^{k},c)$ as
     \begin{align*}
         S(p^{k},c)= \sum_{j=0}^{k}\sum_{a (\mbox{\tiny{mod }} p^{k-j})}^* \sum_{x_0 \in {(\mathbb{Z}/p^{k}\mathbb{Z})}^n}  e\left(\frac{a F(x_0)+\frac{c}{p^j}\cdot x_0}{p^{k-j}} \right).
     \end{align*}
    By replacing $x_0$ by $ x_0+p^{k-j}y,$ for some fixed non-zero $y  \in (\mathbb{Z}/p^j )^n$, the most inner sum vanishes unless $p^j | c$. 
     Indeed, 
    \begin{align*}
    &\sum_{x_0 \in {(\mathbb{Z}/p^{k}\mathbb{Z})}^n}  e\left(\frac{a F(x_0)+\frac{c}{p^j}\cdot x_0}{p^{k-j}} \right) \\
             &=  \sum_{x_0+p^{k-j}y \in {(\mathbb{Z}/p^{k}\mathbb{Z})}^n}  e\left(\frac{a F(x_0)+\frac{c}{p^j}\cdot x_0}{p^{k-j}}   \right) e\left(\frac{c \cdot y}{p^j}\right) \\
       &= e\left(\frac{c \cdot y}{p^j}\right)\sum_{x_0\in {(\mathbb{Z}/p^{k}\mathbb{Z})}^n}  e\left(\frac{a F(x_0)+\frac{c}{p^j}\cdot x_0}{p^{k-j}}   \right). 
    \end{align*}
    Here, in the last equality, we use the fact that each component of $x_0$ and $x_0+p^{k-j}y$ runs over a complete  residue system modulo $p^k$.
\end{proof}
Before we state two main theorems of this section, we recall some facts related to the Gauss character sum (see \cite{B}) 
that will be used often in the proofs of the theorems. Let $\chi$ be a non-trivial Dirichlet character modulo $k$ with conductor $l$. For integer $m$, the Gauss character sum $\tau_k(m,\chi)$ is defined by
        $$\tau_k(m,\chi)=\sum_{n=0}^{k-1} \chi(n)\, e\left(\frac{m\, n}{k}\right),$$
        and one gets (see \cite{B}, Section 1.6) \begin{align*}
            \tau_k(m,\chi)=\frac{k}{l} \, \tau_l\left(\frac{l\,m}{k},\chi\right) \, \mathds{1} _{k|(lm)} \tag{3.2} \label{eq3.2}
        \end{align*} If $\chi$ is a primitive character modulo $k$, then one has (see \cite{IK}, Lemma 3.2)
        \begin{align*}            \tau_k(m,\chi)=\overline{\chi}(m)\, \tau_k(1,\chi), \tag{3.3} \label{eq3.3}
        \end{align*}
          where 
   $$\tau_k(1,\chi)=\tau(\chi)=\sum_{n=0}^{k-1} \chi(n) \,e \left ( \frac{n}{k} \right ).$$
       Moreover, for odd prime $p \nmid m$, one has (see \cite{B}, Theorem 1.5.1) \begin{align*}
           \sum_{n=0}^{p-1} \left( \frac{n}{p}\right )_L e\left(\frac{m\,n}{p}\right) = \sum_{n=0}^{p-1} e \left ( \frac{m\, n^2}{p} \right ), \tag{3.4} \label{eq3.4}
       \end{align*}
       where $\left( \frac{\cdot}{p}\right )_L$ is the Legendre symbol.
       
We are now in a position to prove two main theorems, Theorems \ref{Thm 3.4} and \ref{Thm 3.5}, of this section, accounting for the odd prime and even prime cases respectively. 
\begin{thm}
\label{Thm 3.4}
If $c \, \neq \, 0$ and $p$ is an odd prime, one has
\begin{align*}
     S (p^{k},c)
            = & \,\mu \sum_{l=0}^{\lfloor (k-1)/2 \rfloor}  p^{(2-n)l}  \, \left(\frac{c^t J^{-1} c/ p^{2(k-l)-2} }{p} \right )_L \, \mathds{1} _{p^{k-(2l+1)}|c} \,   \, \mathds{1} _{p^{2(k-l)-2}|c^t J^{-1} c}+ \\ & - p^{n k -1}\sum_{l=1}^{\lfloor k/2 \rfloor} p^{(2-n)l}\,   \, \mathds{1} _{p^{k-2l}|c} \,   \, \mathds{1} _{p^{2(k-l)-1}\,||\,c^t J^{-1} c}+ \\&+ p^{n k -1}(p-1)\sum_{l=1}^{\lfloor k/2 \rfloor} p^{(2-n)l}\,   \, \mathds{1} _{p^{k-2l}|c} \,   \, \mathds{1} _{p^{2(k-l)}|c^t J^{-1} c} +p^{nk} \,  \mathds{1} _{p^{k}|c},
               \end{align*}  
               where $\mu := p^{nk -n/2+1/2} \, i^{(n-1)( p-1)^2/4} \,\left(\frac{ \det J}{p} \right )_L$.
\end{thm}
\begin{proof}
Using Lemma \ref{Lem 3.1}, Proposition \ref{Prop 3.2}, and periodicity of $e(\cdot/p^{k-j})$, we get
\begin{align*}
    S(p^{k},c)&= \sum_{j=0}^{k}\sum_{a (\mbox{\tiny{mod }} p^{k-j})}^* \sum_{x_0 \in {(\mathbb{Z}/p^{k}\mathbb{Z})}^n}  e\left(\frac{a F(x_0)+\frac{c}{p^j}\cdot x_0}{p^{k-j}} \right) \, \mathds{1} _{p^j|c}
    \\
    &= \sum_{j=0}^{k-1} p^{nj}\sum_{a (\mbox{\tiny{mod }} p^{k-j})}^* \sum_{x_0 \in {(\mathbb{Z}/p^{k-j}\mathbb{Z})}^n}  e\left(\frac{a F(x_0)+\frac{c}{p^j}\cdot x_0}{p^{k-j}} \right) \, \mathds{1} _{p^j|c}+p^{nk} \, \mathds{1} _{p^k|c} 
     \\
    &=  \sum_{j=0}^{k-1} p^{nj}\sum_{a (\mbox{\tiny{mod }} p^{k-j})}^* p^{n(k-j)/2} \, G_{k-j}(2aJ,c/p^j)\, \mathds{1} _{p^j|c}+p^{nk} \, \mathds{1} _{p^k|c}\\
    &= \sum_{j=0}^{k-1} p^{nj}\sum_{a (\mbox{\tiny{mod }} p^{k-j})}^* p^{n(k-j)/2} \, e\left( \frac{-a\,c_{p,j,J}}{p^{k-j}} \right)\, \mathds{1} _{p^j|c} \, G_{k-j}(2aJ)+p^{nk} \, \mathds{1} _{p^k|c},
\end{align*}
where $c_{p,j,J}:=(c/p^j)^t J^{-1} (c/p^j)$, and $ G_{k-j}(2aJ)$ in turn can be simplified by using Lemma \ref{Lem 3.1} together with the well-known quadratic Gauss sum (see \cite{B}, Theorem 1.5.2). Specifically, 
\begin{align*}
        G_{k-j}(2aJ)&= \prod_{r=1}^n G_{k-j}(2ab_r) 
        =\prod_{r=1}^n p^{-(k-j)/2}\sum_{x ( \mbox {\tiny{mod}} \, p^{k-j} )} e\left (\frac{a\, b_r \,x^2}{p^{k-j}}   \right ) \\
        &= \prod_{r=1}^n p^{-(k-j)/2} \left(\frac{ab_r}{p^{k-j}} \right )_L p^{(k-j)/2} \, i^{(p^{k-j}-1)^2/4} 
        = \left(\frac{a^n \det J}{p^{k-j}} \right )_L  i^{n( p^{k-j}-1)^2/4}, 
    \end{align*}
    where $(\frac{\cdot}{p^{k-j}})_L$ actually means the Jacobi symbol. Thus,
\begin{align*}
    S(p^{k},c) = \sum_{j=0}^{k-1} p^{n(k+j)/2} \,  i^{n( p^{k-j}-1)^2/4} \, \left(\frac{ \det J}{p^{k-j}} \right )_L \, \overline{X_1}  \, \mathds{1} _{p^j|c} +p^{nk} \, \mathds{1} _{p^k|c}, \tag{3.5} \label{eq3.5}
\end{align*}
in which
$$X_1=\sum_{a (\mbox{\tiny{mod }} p^{k-j})}^*  \left(\frac{a }{p^{k-j}} \right )_L   \, e\left( \frac{a\, c_{p,j,J}}{p^{k-j}} \right).$$

If $2 \nmid (k-j)$, we use (\ref{eq3.2})-(\ref{eq3.4}) to obtain
\begin{align*}
X_1 &=p^{k-j-1}\sum_{a (\mbox{\tiny{mod }} p)} \left(\frac{a}{p} \right )_L e\left (\frac{a\,( p \, c_{p,j,J} / p^{k-j}) }{p}\right ) \, \mathds{1} _{p^{k-j} | ( p\, c_{p,j,J} )}\\   
&= p^{k-j-1/2}\overline{\left(\frac{(p\, c_{p,j,J} / p^{k-j}) }{p} \right )_L} i^{(p-1)^2/4} \, \mathds{1} _{p^{k-j} | ( p\, c_{p,j,J} )},
                     \end{align*}
Since $p^{k-j} \, \equiv \, p \mbox { (mod } 4 \mbox{)}$ with $ 2 \nmid (k-j)$ and the Legendre symbol is of order $2$, the summation in $S(p^{k},c)$ (\ref{eq3.5}) under the assumption 
$2\nmid (k-j)$ is equal to 
\begin{align*}
                     &i^{(n-1)( p-1)^2/4} \,\left(\frac{ \det J}{p} \right )\sum_{j=0}^{k-1} p^{n(k+j)/2} \,    p^{k-j-1/2}   \left(\frac{p\, c_{p,j,J} / p^{k-j} }{p} \right )_L  \, \mathds{1} _{2\nmid k-j}\, \mathds{1} _{p^j|c} \, \mathds{1} _{p^{k-j} | ( p\, c_{p,j,J} )}\\
                        &= \mu \sum_{l=0}^{\lfloor (k-1)/2 \rfloor}  p^{(2-n)l}  \, \left(\frac{c^t J^{-1} c/ p^{2(k-l)-2} }{p} \right )_L \, \mathds{1} _{p^{k-(2l+1)}|c} \,  \mathds{1} _{p^{2(k-l)-2}|c^t J^{-1} c},
        \end{align*}  
where $\mu = p^{nk -n/2+1/2} \, i^{(n-1)( p-1)^2/4} \,\left(\frac{ \det J}{p} \right )_L$. We obtain the first term in Theorem \ref{Thm 3.4}. 

Now, we consider $(k-j) > 0$ even. In this case, 
$p^{k-j} \, \equiv \, 1 \mbox {(mod } 4 \mbox{)}$ and then the summation in $S(p^{k},c)$ (\ref{eq3.5}) becomes 
\begin{align*}
        \sum_{j=0}^{k-1} p^{n(k+j)/2} \,   \, \mathds{1} _{p^j|c} \, \mathds{1} _{2|k-j} \sum_{a (\mbox{\tiny{mod }} p^{k-j})}^*    \, e\left( \frac{-a\, c_{p,j,J}}{p^{k-j}} \right),
\end{align*}
where the real-valued Ramanujan sum (see \cite{IK}, Section 3.2) applied to prime-powers yields
\begin{align*}
                \sum_{a (\mbox{\tiny{mod }} p^{k-j})}^{*}  e\left( \frac{-a\, c_{p,j,J}}{p^{k-j}} \right)
  &= \left\{ \begin{array}{cll} 
    0 & \mbox{if } p^{k-j-1} \nmid \, c_{p,j,J}    \\ \\
                         -p^{k-j-1}  & \mbox{if } p^{k-j-1} \,||\, \, c_{p,j,J}       \\ \\
                            \phi(p^{k-j})=p^{k-j-1}(p-1) & \mbox{if } p^{k-j}|\, c_{p,j,J} \end{array}\right. ,
                          \end{align*}
where $\phi$ is the Euler's totient function.
 Thus, we obtain the following two middle terms of $ S(p^{k},c)$ in Theorem \ref{Thm 3.4}: 
  \begin{align*}
        &\sum_{j=0}^{k-1} p^{n(k+j)/2} \,   \, \mathds{1} _{p^j|c} \, \mathds{1} _{2|k-j} \left[(-p^{k-j-1})  \,   \, \mathds{1} _{p^{k-j-1} \,||\, \, c_{p,j,J}}+p^{k-j-1}(p-1)\,   \, \mathds{1} _{p^{k-j}|\, c_{p,j,J}} \right]\\
        &=- p^{n k -1}\sum_{l=1}^{\lfloor k/2 \rfloor} p^{(2-n)l}\,   \, \mathds{1} _{p^{k-2l}|c} \,   \, \mathds{1} _{p^{2(k-l)-1}\,||\,c^tJ^{-1}c} + \\&+  p^{n k -1}(p-1)\sum_{l=1}^{\lfloor k/2 \rfloor} p^{(2-n)l}\,   \, \mathds{1} _{p^{k-2l}|c} \,   \, \mathds{1} _{p^{2(k-l)}|c^tJ^{-1}c}. 
 \end{align*}
Hence, the theorem follows.
\end{proof}
In particular, by considering some specific values of $c$, we apply Theorem \ref{Thm 3.4} to deduce the following corollaries. 
\begin{customcor}{${3.4.1}$}
\label{Cor 3.4.1}
\textit{If $c= 0$ and $p$ is an odd prime, one has
 \begin{align*}
   \qquad \qquad \qquad \qquad  \qquad S(p^{k},c)
    =p^{nk-1}(p-1)\sum_{l=0}^{\lfloor k/2 \rfloor} p^{(2-n)l}+ p^{nk}. \hspace{1.8 in} \square
\end{align*}  }  
\end{customcor}

\begin{customcor}{${3.4.2}$}
\label{Cor 3.4.2}
\textit{
If $c \, \neq \, 0$, $c^t J c =0$, and $p$ is an odd prime, one has
 \begin{align*}
   \qquad \qquad \qquad S(p^{k},c)
    =p^{nk-1}(p-1)\sum_{l=0}^{\lfloor k/2 \rfloor} p^{(2-n)l} \,   \, \mathds{1} _{p^{k-2l}|c}+ p^{nk}\,  \mathds{1} _{p^{k}|c}. \hspace{1.6 in} \square
\end{align*} }
\end{customcor}
We assumed $p$ is odd in Theorem \ref{Thm 3.4}. We also need a companion statement when $p=2$ in Theorem \ref{Thm 3.5}. In fact, for $p=2$, the quadratic Gauss sum yields an extra term of $(1+i^{\#})$, for some variable $\#$, (see \cite{B}, Theorem 1.5.4), which generally makes computation related to this case more irritating and complicated. 
To claim Theorem \ref{Thm 3.5}, we first prove the following Lemmas \ref{Lem 3.5.1} and \ref{Lem 3.5.2}.
        \begin{customlem}{${3.5.1}$}
        \label{Lem 3.5.1}
        \textit{
         \begin{align*} 
     \sum_{a (\mbox{\tiny{mod }} 2^{k-j})}^*\,  e\left( \frac{-a\,c_{2,j,J}}{2^{k-j}} \right)\, \prod_{r=1}^n (1+i^{ab_r}) = d_{n,j,k,J} \, 2^{k-j-1},
    \end{align*}
   where
    \begin{align*} d_{n,j,k,J}
     =\left\{ \begin{array}{cll} 
    0 & \mbox{if } 2^{k-j-2} \nmid \,c_{2,j,J}    \\ \\
                         - \, \chi_4(\,c_{2,j,J}/2^{k-j-2}) \,  2^{n/2} \sin \left(\frac{(2m-n) \pi}{4}\right) & \mbox{if } 2^{k-j-2} \,||\, \,c_{2,j,J}       \\ \\
                           - \,  2^{n/2} \cos \left(\frac{(2m-n) \pi}{4}\right) & \mbox{if } 2^{k-j-1} \,||\, \,c_{2,j,J}       \\ \\  2^{n/2} \cos \left(\frac{(2m-n) \pi}{4}\right) & \mbox{if } 2^{k-j} | \,c_{2,j,J}  \end{array}\right.  , \tag{3.6} \label{eq3.6}
    \end{align*}
    and $\chi_4$ is the real Dirichlet character modulo $4$ defined by
    \begin{align*}
         \chi_4(a)
  = \left\{ \begin{array}{cll} 
    1 & \mbox{if } a\, \equiv \, 1 \,\mbox {(mod } 4  \mbox {)}    \\ \\
                           -1 & \mbox{if } a\, \equiv \, 3 \, \mbox {(mod } 4\mbox {).} 
                            \\ \\
                           0 & \mbox{if } 2|a \end{array}\right.
                          \end{align*} }
                                  \end{customlem} 
  \begin{proof}
   By setting $m=\#\{b_r: b_r=1, r=1,\ldots,n\}$, $m=1,\ldots, n-1,$ and switching sides of the quadratic form as necessary, we assume $m > n-m$, i.e., $2m > n$. Then,
\begin{align*}
       \prod_{r=1}^n (1+i^{ab_r}) = 2^{n-m} (1+i^a)^{2m-n}. 
       \end{align*}
Since $i^a$ depends only on $a$ modulo $4$, we get
\begin{align*} 
\sum_{a (\mbox{\tiny{mod }} 2^{k-j})}^*\,  e\left( \frac{-a\,c_{2,j,J}}{2^{k-j}} \right)\, \prod_{r=1}^n (1+i^{ab_r})= 2^{n-m} \,  [ I_+\,  (1+i)^{2m-n} + I_-\,  (1-i)^{2m-n} ],
    \end{align*}
    where we write
    \begin{align*}
        I_{+}:=\sum_{\substack{a (\mbox{\tiny{mod }} 2^{k-j})\\ a \, \equiv \, 1 (\mbox{\tiny{mod }} 4)}}^{*} \,  e\left( \frac{-a\,c_{2,j,J}}{2^{k-j}} \right) \mbox { and }
        I_{-}:=\sum_{\substack{a (\mbox{\tiny{mod }} 2^{k-j})\\ a \, \equiv \,  3 (\mbox{\tiny{mod }} 4)}}^{*} \,  e\left( \frac{-a\,c_{2,j,J}}{2^{k-j}} \right).
    \end{align*}
    While the sum of $I_+$ and $I_-$ comes down to the Ramanujan sum, their difference turns to be a Gauss character sum. Specifically,
    \begin{align*}
                                        I_+ - I_- =& \displaystyle{\sum_{a (\mbox{\tiny{mod }} 2^{k-j})}^*}\, \chi_4(a) \, e\left( \frac{-a\,c_{2,j,J}}{2^{k-j}} \right)
                           ,  \mbox {where }  \chi_4(a)
  = \left\{ \begin{array}{cll} 
    1 & \mbox{if } a\, \equiv \, 1 \,\mbox {(mod } 4  \mbox {)}    \\ \\
                           -1 & \mbox{if } a\, \equiv \, 3 \, \mbox {(mod } 4\mbox {).} 
                            \\ \\
                           0 & \mbox{if } 2|a \end{array}\right.
                          \end{align*}
                         We note that $\chi_4$ is primitive. Using (\ref{eq3.2}) and (\ref{eq3.3}), this Gauss character sum can be explicitly evaluated, and so  \begin{align*}
                I_+ - I_- =2^{k-j-1} \, \chi_4(\,c_{2,j,J}/2^{k-j-2}) \, (-i).
    \end{align*}    
Moreover, the real-valued Ramanujan sum yields
   \begin{align*}
              \sum_{a (\mbox{\tiny{mod }} 2^{k-j})}^*\,  e\left( \frac{-a\,c_{2,j,J}}{2^{k-j}} \right)= \sum_{a (\mbox{\tiny{mod }} 2^{k-j})}^*\,  e\left( \frac{a\,c_{2,j,J}}{2^{k-j}} \right) &= \left\{ \begin{array}{cll} 
    0 & \mbox{if } 2^{k-j-1} \nmid \,c_{2,j,J}    \\ \\
                         -2^{k-j-1}  & \mbox{if } 2^{k-j-1} \,||\, \,c_{2,j,J}       \\ \\
                            \phi(2^{k-j})=2^{k-j-1} & \mbox{if } 2^{k-j} | \,c_{2,j,J}  \end{array}\right. . \end{align*}    
Thus,
  \begin{align*}
        I_{\pm}
            &= \frac{1}{2} \left [\sum_{a (\mbox{\tiny{mod }} 2^{k-j})}^*\,  e\left( \frac{-a\,c_{2,j,J}}{2^{k-j}} \right)\pm \sum_{a (\mbox{\tiny{mod }} 2^{k-j})}^*\, \chi_4(a) \, e\left( \frac{-a\,c_{2,j,J}}{2^{k-j}} \right)\right ] \\
        &=  \left\{ \begin{array}{cll} 
    0 & \mbox{if } 2^{k-j-2} \nmid \,c_{2,j,J}    \\ \\
                         \mp2^{k-j-2} \chi_4(\,c_{2,j,J}/2^{k-j-2}) \, i  & \mbox{if } 2^{k-j-2} \,||\, \,c_{2,j,J}       \\ \\
                           -2^{k-j-2}  & \mbox{if } 2^{k-j-1} \,||\, \,c_{2,j,J}       \\ \\ 2^{k-j-2} & \mbox{if } 2^{k-j} | \,c_{2,j,J}  \end{array}\right. ,
    \end{align*}
    which implies
      \begin{align*} 
         &\sum_{a (\mbox{\tiny{mod }} 2^{k-j})}^*\,  e\left( \frac{-a\,c_{2,j,J}}{2^{k-j}} \right)\, \prod_{r=1}^n (1+i^{ab_r})=
    \\&=2^{n-m}   \left\{ \begin{array}{cll} 
    0 & \mbox{if } 2^{k-j-2} \nmid \,c_{2,j,J}    \\ \\
                         2^{k-j-2} \, \chi_4(\,c_{2,j,J}/2^{k-j-2}) \, i\, 2i\, \mbox{ Im(} (1+i)^{2m-n} \mbox {)} & \mbox{if } 2^{k-j-2} \,||\, \,c_{2,j,J}       \\ \\
                           -2^{k-j-2} \, 2 \mbox{ Re(} (1+i)^{2m-n} \mbox {)} & \mbox{if } 2^{k-j-1} \,||\, \,c_{2,j,J}       \\ \\ 2^{k-j-2}\, 2 \mbox{ Re(} (1+i)^{2m-n} \mbox {)} & \mbox{if } 2^{k-j} | \,c_{2,j,J}  \end{array}\right .\\
                           &=\left\{ \begin{array}{cll} 
    0 & \mbox{if } 2^{k-j-2} \nmid \,c_{2,j,J}    \\ \\
                         -2^{k-j-1} \, \chi_4(\,c_{2,j,J}/2^{k-j-2}) \,  2^{n/2} \sin \left(\frac{(2m-n) \pi}{4}\right) & \mbox{if } 2^{k-j-2} \,||\, \,c_{2,j,J}       \\ \\
                           -2^{k-j-1} \,  2^{n/2} \cos \left(\frac{(2m-n) \pi}{4}\right) & \mbox{if } 2^{k-j-1} \,||\, \,c_{2,j,J}       \\ \\ 2^{k-j-1}\, 2^{n/2} \cos \left(\frac{(2m-n) \pi}{4}\right) & \mbox{if } 2^{k-j} | \,c_{2,j,J}  \end{array}\right .  .  \end{align*}
  \end{proof}
  
\begin{customlem}{${3.5.2}$}
\label{Lem 3.5.2}
\textit{
     \begin{align*}
    \sum_{a (\mbox{\tiny{mod }} 2^{k-j})}^*\, \left (\frac{2}{a } \right )_L \, e\left( \frac{-a \,c_{2,j,J}}{2^{k-j}} \right)\, \prod_{r=1}^n (1+i^{ab_r}) = e_{n,j,k,J} \, 2^{k-j-3/2} \,  \mathds{1} _{2^{k-j-3} \,||\, \,c_{2,j,J} },
\end{align*}
where
\begin{align*}         
    e_{n,j,k,J} = 2^{n/2} \Big[ \chi_{8,2}\left(\frac{\,c_{2,j,J}}{2^{k-j-3}}\right)\sin \left(\frac{(2m-n) \pi}{4}\right)+\chi_{8,3}\left(\frac{\,c_{2,j,J}}{2^{k-j-3}}\right)\cos \left(\frac{(2m-n) \pi}{4}\right)\Big],   \tag{3.7} \label{eq3.7}    
    \end{align*}
    and $\chi_{8,2}, \chi_{8,3}$  are the real Dirichlet characters modulo  $8$ satisfying
   \begin{align*}
        \chi_{8,2}(a)= \chi_{8,3}(a)=0 \mbox { for } (a,8) \, \neq \, 1, \mbox { and } \quad
        \begin{aligned}[c]
                  \begin{tabular}{ |c|c|c| c|c|}  \hline
a \mbox{(\small{mod} } 8\mbox{)} &1 & 3 & 5 & 7 \\ \hline
  $\chi_{8,2}(a)$ & 1 & 1 & -1 &-1 \\ \hline
   $\chi_{8,3}(a)$ & 1 & -1 & -1 &1\\ \hline
\end{tabular} .
   \end{aligned}
   \end{align*} }
     \end{customlem}
     
  \begin{proof}
  Similar to Lemma \ref{Lem 3.5.1}, we get
\begin{align*} 
&\sum_{a (\mbox{\tiny{mod }} 2^{k-j})}^*\, \left (\frac{2}{a } \right )_L \, e\left( \frac{-a\,c_{2,j,J}}{2^{k-j}} \right)\, \prod_{r=1}^n (1+i^{ab_r})
\\&= 2^{n-m}  \left[ I_1\,  (1+i)^{2m-n} + I_7\,  (1-i)^{2m-n}-I_3\,  (1-i)^{2m-n} - I_5\,  (1+i)^{2m-n} \right ],
    \end{align*}
    in which 
    \begin{align*}
        I_{i}:=\sum_{\substack{a (\mbox{\tiny{mod }} 2^{k-j})\\ a \,  \equiv \,i (\mbox{\tiny{mod }} 8)}}^{*} \,  e\left( \frac{-a\,c_{2,j,J}}{2^{k-j}} \right) \mbox { with }
         i=1, 3, 5, \mbox { and } 7.
    \end{align*}
    We observe that
    \begin{align*}
                  \left\{ \begin{array}{cll} 
                              I_1 - I_7-I_3 + I_5 &=  \displaystyle{\sum_{a (\mbox{\tiny{mod }} 2^{k-j})}^*}\, \chi_{8,1}(a) \, e\left( \frac{-a\,c_{2,j,J}}{2^{k-j}} \right) =:A_1  \\ \\
                              I_1 - I_7+I_3 - I_5 &=  \displaystyle{\sum_{a (\mbox{\tiny{mod }} 2^{k-j})}^*}\, \chi_{8,2}(a) \, e\left( \frac{-a\,c_{2,j,J}}{2^{k-j}} \right) =:A_2  \\ \\
                         I_1 + I_7-I_3 - I_5 &=  \displaystyle{\sum_{a (\mbox{\tiny{mod }} 2^{k-j})}^*}\, \chi_{8,3}(a)  e\left( \frac{-a\,c_{2,j,J}}{2^{k-j}} \right)=:A_3   \\ \\
                           I_1 + I_7+I_3 + I_5 &=  \displaystyle{\sum_{a (\mbox{\tiny{mod }} 2^{k-j})}^*}\,  e\left( \frac{-a\,c_{2,j,J}}{2^{k-j}} \right) =:A_4      \end{array}\right.,                           \end{align*}
                          where  $\chi_{8,1}, \chi_{8,2}, \chi_{8,3}$  are the real Dirichlet characters modulo  $8$ satisfying          \begin{align*}
       \chi_{8,1}(a)= \chi_{8,2}(a)= \chi_{8,3}(a)=0 \mbox { for } (a,8) \, \neq \, 1 \mbox { and } \quad
        \begin{aligned}[c]
                  \begin{tabular}{ |c|c|c| c|c|}  \hline
a \mbox {(\small{mod} } 8\mbox {)}&1 & 3 & 5 & 7 \\ \hline
 $\chi_{8,1}(a)$ & 1 & -1 & 1 &-1 \\ \hline
  $\chi_{8,2}(a)$ & 1 & 1 & -1 &-1 \\ \hline
   $\chi_{8,3}(a)$ & 1 & -1 & -1 &1\\ \hline
\end{tabular} .
   \end{aligned}
     \end{align*}
     Moreover, $A_4$ is the Ramanujan sum, and applying (\ref{eq3.2}) and (\ref{eq3.3}), we have
  \begin{align*}
       A_1    =2^{k-j-3} \, \chi_{8,1}(\,c_{2,j,J}/2^{k-j-3}) \, \overline{\sum_{a (\mbox{\tiny{mod }} 8)}^*\, \chi_{8,1}(a) \,e\left( \frac{a}{8} \right)}=0,
    \end{align*}      
  \begin{align*}
        A_2       =2^{k-j-3} \, \chi_{8,2}(\,c_{2,j,J}/2^{k-j-3}) \, \overline{\sum_{a (\mbox{\tiny{mod }} 8)}^*\, \chi_{8,2}(a) \,e\left( \frac{a}{8} \right)}
        = 2^{k-j-3/2} \, \chi_{8,2}(\,c_{2,j,J}/2^{k-j-3}) \, (-i),
    \end{align*}          
      and                       
  \begin{align*}
       A_3
          &=2^{k-j-3} \, \chi_{8,3}(\,c_{2,j,J}/2^{k-j-3}) \, \overline{\sum_{a (\mbox{\tiny{mod }} 8)}^*\, \chi_{8,3}(a) \,e\left( \frac{a}{8} \right)}
        = 2^{k-j-3/2} \, \chi_{8,3}(\,c_{2,j,J}/2^{k-j-3}) .
    \end{align*} 
    Thus,
 \begin{equation*}
\begin{aligned}[c]
 I_1&=(1/4)(A_1+A_2+A_3+A_4)  \\
        &=  \left\{ \begin{array}{cll} 
    0 & \mbox{if } 2^{k-j-3} \nmid \,c_{2,j,J}    \\ \\
                         (1/4)(A_2+A_3)  & \mbox{if } 2^{k-j-3} \,||\, \,c_{2,j,J}       \\ \\
                           (1/4)A_4  & \mbox{if } 2^{k-j-1} \,||\, \,c_{2,j,J}       \\ \\ (1/4)A_4 & \mbox{if } 2^{k-j} | \,c_{2,j,J}  \end{array}\right .,
\end{aligned}
\qquad\qquad
\begin{aligned}[c]
I_7&=(1/4)(-A_1-A_2+A_3+A_4)  \\
        &=  \left\{ \begin{array}{cll} 
    0 & \mbox{if } 2^{k-j-3} \nmid \,c_{2,j,J}    \\ \\
                         (1/4)(-A_2+A_3)  & \mbox{if } 2^{k-j-3} \,||\, \,c_{2,j,J}       \\ \\
                           (1/4)A_4  & \mbox{if } 2^{k-j-1} \,||\, \,c_{2,j,J}       \\ \\ (1/4)A_4 & \mbox{if } 2^{k-j} | \,c_{2,j,J}  \end{array}\right .,
\end{aligned}
\end{equation*}
    \begin{align*}
\begin{aligned}[c]
 I_3&=(1/4)(-A_1+A_2-A_3+A_4)  \\
        &=  \left\{ \begin{array}{cll} 
    0 & \mbox{if } 2^{k-j-3} \nmid \,c_{2,j,J}    \\ \\
                         (1/4)(A_2-A_3)  & \mbox{if } 2^{k-j-3} \,||\, \,c_{2,j,J}       \\ \\
                           (1/4)A_4  & \mbox{if } 2^{k-j-1} \,||\, \,c_{2,j,J}       \\ \\ (1/4)A_4 & \mbox{if } 2^{k-j} | \,c_{2,j,J}  \end{array}\right ., \mbox { and} \qquad \end{aligned}
\begin{aligned}[c]
I_5&=(1/4)(A_1-A_2-A_3+A_4)   \\
        &=  \left\{ \begin{array}{cl} 
    0 & \mbox{if } 2^{k-j-3} \nmid \,c_{2,j,J}    \\ \\
                         (1/4)(-A_2-A_3)  & \mbox{if } 2^{k-j-3} \,||\, \,c_{2,j,J}       \\ \\
                           (1/4)A_4  & \mbox{if } 2^{k-j-1} \,||\, \,c_{2,j,J}       \\ \\ (1/4)A_4 & \mbox{if } 2^{k-j} | \,c_{2,j,J}  \end{array}\right .,
\end{aligned}
\end{align*}
    which ultimately yields
      \begin{align*}
      &\sum_{a (\mbox{\tiny{mod }} 2^{k-j})}^*\, \left (\frac{2}{a } \right )_L \, e\left( \frac{-a\,c_{2,j,J}}{2^{k-j}} \right)\, \prod_{r=1}^n (1+i^{ab_r})=
    \\&=2^{n/2} \left\{ \begin{array}{cll} 
    0 & \mbox{if } 2^{k-j-3} \nmid \,c_{2,j,J}    \\ \\  2^{k-j-3/2} \, [ \chi_{8,2}(\frac{\,c_{2,j,J}}{2^{k-j-3}})\sin (\frac{(2m-n) \pi}{4})+\chi_{8,3}(\frac{\,c_{2,j,J}}{2^{k-j-3}})\cos (\frac{(2m-n) \pi}{4})] & \mbox{if } 2^{k-j-3} \,||\, \,c_{2,j,J}       \\ \\    0& \mbox{if } 2^{k-j-1} \,||\, \,c_{2,j,J}       \\ \\ 0 & \mbox{if } 2^{k-j} | \,c_{2,j,J}  \end{array}\right . \\&= e_{n,j,k,J} \, 2^{k-j-3/2} \, \mathds{1} _{2^{k-j-3} \,||\, \,c_{2,j,J} },   \end{align*}
                           where
                           \begin{align*}         
    e_{n,j,k,J} = 2^{n/2} \Big[ \chi_{8,2}\left(\frac{\,c_{2,j,J}}{2^{k-j-3}}\right)\sin \left(\frac{(2m-n) \pi}{4}\right)+\chi_{8,3}\left(\frac{\,c_{2,j,J}}{2^{k-j-3}}\right)\cos \left(\frac{(2m-n) \pi}{4}\right)\Big].      \end{align*}
  \end{proof}
  
        Combining Lemmas \ref{Lem 3.5.1} and \ref{Lem 3.5.2} together, we deduce Theorem \ref{Thm 3.5}:
        \begin{thm}
        \label{Thm 3.5}
        If $c \, \neq \, 0$ and $p=2$, one has
        \begin{align*}
    S(2^{k},c)= &\sum_{j=0}^{k-2}  2^{j(n/2-1)+k(n/2+1)-1} d_{n,j,k,J} \,  \mathds{1} _{2^j|c}\, \mathds{1} _{2|k-j} \,  \,  +\\&+ \sum_{j=0}^{k-2}  2^{j(n/2-1)+k(n/2+1)-3/2} \, \left (\frac{2}{\det J} \right )_L\, e_{n,j,k,J} \,  \mathds{1} _{2^j|c} \, \mathds{1} _{2\nmid k-j} \, \mathds{1} _{2^{k-j-3} \,||\, \,c_{2,j,J} } + 2^{nk} \, \mathds{1} _{2^{k}|c},
\end{align*}
where $d_{n,j,k,J}$ (\ref{eq3.6}) and $ e_{n,j,k,J}$ (\ref{eq3.7}) are constants depending on $n, j, k$, and $J$. 
        \end{thm}
  \begin{proof} 
  We proceed similarly as the case $p \, \neq \, 2$ and get
\begin{align*}
    S(2^{k},c)
    &= \sum_{j=0}^{k-1} 2^{nj}\sum_{a (\mbox{\tiny{mod }} 2^{k-j})}^* 2^{n(k-j)/2} \, e\left( \frac{-a\,c_{2,j,J}}{2^{k-j}} \right)\, \mathds{1} _{2^j|c} \, G_{k-j}(2aJ)+2^{nk} \, \mathds{1} _{2^{k}|c}.
\end{align*}
where
\begin{align*}
    G_{k-j}(2aJ)&=\prod_{r=1}^n G_{k-j}(2ab_r)     = \prod_{r=1}^n 2^{-(k-j)/2}\sum_{x ( \mbox {\tiny{mod}} \, 2^{k-j} )} e\left (\frac{a\,b_{r}\,x^2}{2^{k-j}}   \right ).       \end{align*}    
\begin{itemize}
    \item If $k-j=1$, then Theorem 1.5.1 in \cite{B} yields
        \begin{align*}
    \sum_{x ( \mbox {\tiny{mod}} \, 2 )} e\left (\frac{a\,b_{r}\,x^2}{2}   \right )=0, \mbox { which implies }G_{k-j}(2aJ)=0. 
        \end{align*}  
      \item If $k-j > 1$, then Lemma \ref{Lem 3.1} and Theorem 1.5.4 in \cite{B} give us 
           \begin{align*}
    G_{k-j}(2aJ)&=\prod_{r=1}^n G_{k-j}(2ab_r) 
    = \prod_{r=1}^n 2^{-(k-j)/2}\sum_{x ( \mbox {\tiny{mod}} \, 2^{k-j} )} e\left (\frac{a\,b_r\,x^2}{2^{k-j}}   \right ) \\
    &= \prod_{r=1}^n 2^{-(k-j)/2} \left (\frac{2^{k-j}}{a\,b_r} \right )_L 2^{k-j/2}    (1+i^{ab_r}) 
    = \left (\frac{2^{k-j}}{a^n \,\det J} \right )_L \prod_{r=1}^n (1+i^{ab_r}).
    \end{align*}    
\end{itemize}
Thus, $S(2^{k},c)$ becomes
\begin{align*}
      \sum\limits_{j=0}^{k-2}  2^{n(k+j)/2} \left (\frac{2^{k-j}}{\det J} \right )_L \sum\limits_{a (\mbox{\tiny{mod }} 2^{k-j})}^*\, \left (\frac{2^{k-j}}{a } \right )_L \, e\left( \frac{-a\,c_{2,j,J}}{2^{k-j}} \right)\, \prod_{r=1}^n (1+i^{ab_r}) \,  \mathds{1} _{2^j|c} \, +2^{nk} \, \mathds{1} _{2^{k}|c}.
\end{align*}
We consider the parity of $(k-j)$ and use Lemmas \ref{Lem 3.5.1} and \ref{Lem 3.5.2} to complete the proof. 
\end{proof}

As in the case $p \, \neq \, 2$, the following two corollaries are the consequences of Theorem \ref{Thm 3.5} applied to certain values of $c$. Specifically, \begin{customcor}{${3.5.1}$}
\label{Cor 3.5.1} 
\textit{
If $c=0$ and $p=2$, one has
\begin{align*}
    S(2^{k},c)= &\sum_{j=0}^{k-2}  2^{j(n/2-1)+k(n/2+1)-1} d_{n,J} \,  \mathds{1} _{2|k-j} \,  + 2^{nk} ,
\end{align*}
        where 
 \begin{align*}
                d_{n,J}= \left\{ \begin{array}{cll} 
    2^{(n-1)/2} & \mbox{if } 2m-n \, \equiv \, 1 \mbox { or } 7 \mbox { (mod } 8 \mbox {)}    \\ \\
                         -2^{(n-1)/2} & \mbox{if } 2m-n \, \equiv \, 3 \mbox { or } 5 \mbox { (mod } 8 \mbox {)} \quad  
                         \end{array}\right. \tag{3.8} \label{eq3.8}
                          \end{align*} }
                          \end{customcor}
\begin{proof} It follows by applying Theorem $\ref{Thm 3.5}$ to the case $c=0$ and noting that $\displaystyle{\left (\frac{2}{x } \right )_L=0}$ for $(2,x) \, \neq \, 1$ and $\displaystyle{\cos \left(\frac{(2m-n) \pi}{4}\right)=\pm 2^{-1/2}}$ depending only on $2m-n$ modulo $8$. 
\end{proof}
Using the arguments in Corollary \ref{Cor 3.5.1}, we also obtain
\begin{customcor}{${3.5.2}$}
$\label{Cor 3.5.2}$  
\textit{
If $c \, \neq \, 0$, $c^t J^{-1} c=0$ and $p=2$, one has
\begin{align*}
    S(2^{k},c)= &\sum_{j=0}^{k-2}  2^{j(n/2-1)+k(n/2+1)-1} d_{n,J} \,  \mathds{1} _{2^j|c}\, \mathds{1} _{2|k-j} \,  + 2^{nk} \, \mathds{1} _{2^{k}|c},
\end{align*}
where $d_{n,J}$ (\ref{eq3.8}) is defined as in Corollary \ref{Cor 3.5.1}. } \hspace{ 2.69 in} $\square$
 \end{customcor}
 
\section{The Dirichlet series}
\label{Sec 4}
Let $s=\sigma+it$. We recall that the Dirichlet series $D(c,s)$ converges absolutely for $\sigma=\sigma_0 >1$, and it has multiplicative coefficient $S(q,c)$. Therefore, $D(c,s)$ admits the Euler product with Euler's factor $D_p(c,s)$ given in
\begin{align*}
    D(c,s)=\sum_{q=1}^{\infty} \frac{S(q,c)}{q^{s+n+1}}=\prod_p  \sum_{k=0}^{\infty}\frac{S(p^{k},c)}{p^{k(s+n+1)}}=:\prod_p D_p(c,s).
\end{align*}
Having studied $S(p^{k},c)$ in Section \ref{Sec 3}, we are ready to investigate the analytic continuation of $D(c,s)$ regarded as a function of $s$. Indeed, $D(c,s)$ can be written in terms of the most basic and well-known Dirichlet series, the so-called Riemann zeta function, and Dirichlet L-functions. To ease the exposition, we further divide this section into 3 subsections corresponding to the values of $c$ of interest.

\addtocontents{toc}{\protect\setcounter{tocdepth}{1}}
{ \centering
\subsection*{4.1}{The case $c=0$}
\par }

We recall that the summation over $c$ occurring in our secondary term $c_2$ in (\ref{eq1.3}) is the result of an application of Poisson summation. We thereby expect to retrieve the main term in the asymptotic of $N(B)$ in (\ref{eq2.2}) at $c=0$. In this case, it turns out that the Dirichlet series $D(0,s)$ has the form of a convergent infinite product multiplied by a  quotient of $\zeta$ functions, $\frac{\zeta(s+1)\,\zeta(2s+n)}{\zeta(2s+n+1)}$, in the half plane $\sigma > -n/2.$ Consequently, $D(0,s)$ can be meromorphically continued beyond the half plane $\sigma=\sigma_0$ by inheriting the meromorphic continuation of $\zeta$ in the whole complex plane.
\begin{thm}
\label{Thm 4.1}
If $c=0$, one has
\begin{align*}
     D(0,s)= D_2(0,s) \, \frac{(1-2^{-(s+1)})(1-2^{-(2s+n)})}{1-2^{-(2s+n+1)}} \, \frac{\zeta(s+1)\,\zeta(2s+n)}{\zeta(2s+n+1)}. \end{align*}
Moreover,
\begin{align*}
    D_2(0,s) \, \frac{(1-2^{-(s+1)})(1-2^{-(2s+n)})}{1-2^{-(2s+n+1)}} \ll_{n, \epsilon} 1 \mbox { for } \sigma \geq -n/2 + \epsilon.
\end{align*}
\end{thm}
\begin{proof}
By Corollary \ref{Cor 3.4.1}, we have
\begin{align*}
   \prod_{p\, \neq \, 2} D_p(0,s)
    &=  \prod_{p \, \neq \, 2} \left [ \frac{p^{n-1}-1}{p^{n-1}-p} \sum_{k=0}^{\infty}\frac{1  }{p^{k(s+1)}} - \frac{p-1}{p^{n-1}-p} \sum_{k=0}^{\infty} \frac{p^{(2-n)\lfloor k/2 \rfloor}}{p^{k(s+1)}}  \right] \\
     &= \prod_{p \, \neq \, 2} \left [   \frac{1-p^{-(2s+n+1)} }{(1-p^{-(s+1)})(1-p^{-(2s+n)})}  \right],   
\end{align*}
where the last equality is due to
\begin{align*}
    \sum_{k=0}^{\infty} \frac{p^{(2-n)\lfloor k/2 \rfloor}}{p^{k(s+1)}}&=\sum_{l=0}^{\infty} \frac{p^{(2-n)l}}{p^{2l(s+1)}}+\sum_{l=0}^{\infty} \frac{p^{(2-n)l}}{p^{(2l+1)(s+1)}} 
    =\left ( 1+ \frac{1}{p^{s+1}} \right )  \sum_{l=0}^{\infty} \frac{1}{p^{l(2s+n)}}. \end{align*}
Then, we have
\begin{align*}
     D(0,s)&=\prod_p D_p(0,s) =D_2(0,s) \prod_{ p \, \neq \, 2} D_p(0,s) \\
     &= D_2(0,s) \prod_{p \, \neq \, 2} \left [   \frac{1-p^{-(2s+n+1)} }{(1-p^{-(s+1)})(1-p^{-(2s+n)})}  \right] \\
      &= D_2(0,s) \, \frac{(1-2^{-(s+1)})(1-2^{-(2s+n)})}{1-2^{-(2s+n+1)}} \, \frac{\zeta(s+1)\,\zeta(2s+n)}{\zeta(2s+n+1)}. 
\end{align*}
Moreover, Corollary \ref{Cor 3.5.1} yields
\begin{align*}
    D_2(0,s)&= \sum_{k=0}^{\infty} \frac{S(2^{k},0)}{2^{k(s+n+1)}} 
        = \sum_{k=0}^{\infty} \frac{ 2^{nk-1}  d_{n,J} \, +2^{nk}}{2^{k(s+n+1)}}  \sum_{l=1}^{\lfloor k/2 \rfloor} 2^{l(2-n)} 
           \\&= (1+2^{-(s+1)}) \left (\frac{1+a_{n,J}}{1-2^{-2(s+1)}} -  \frac{a_{n,J}}{1-2^{-(2s+n)}} \right ),
\end{align*}
where $\displaystyle{a_{n,J}=\frac{d_{n,J}}{2^{n-1}-2}}$ and $d_{n,J}$ given in (\ref{eq3.8}) is a constant depending on $n$ and $J$ .
Thus, for $\sigma \geq -n/2 + \epsilon$, we obtain
\begin{align*}
&D_2(0,s) \, \frac{(1-2^{-(s+1)})(1-2^{-(2s+n)})}{1-2^{-(2s+n+1)}} 
    = \frac{(1-2^{-(2s+n)}+d_{n,J}\,2^{-(2s+n+1)})}{1-2^{-(2s+n+1)}} \tag{4.1} \label{eq4.1} \\
    & \ll \frac{(1+2^{n/2-1-\epsilon})(1-2^{-2\epsilon}+d_{n,J}\,2^{-(2\epsilon+1)})}{(1-2^{-(2\epsilon+1)})} \ll_{n, \epsilon, J} 1. 
\end{align*}
\end{proof}
{ \centering
\subsection*{4.2}{The case $c \, \neq \, 0$ and $c^t J^{-1} c =0$}
\par }

For the case  $c \in \mathbb{Z}^n -\{0\}$, we further restrict $D(c,s)$ to such $c^{\prime}$s satisfying $c^t J^{-1} c=0$. We again obtain a quotient of Riemann zeta functions, $\frac{\zeta(2s+n)}{\zeta(2s+n+1)}$, as a factor of $D(c,s)$. More precisely, 
\begin{thm}
\label{Thm 4.2}
If $c \, \neq \, 0$ and $c^t J^{-1} c =0$, one has
\begin{align*}
    D(c,s)=D_2(c,s) \, \frac{1-2^{-(2s+n)}} {1-2^{-(2s+n+1)}} \, \frac{\zeta(2s+n)}{\zeta(2s+n+1)} \, \prod_{ p | c, \, p \, \neq \, 2} D_p(c,s) \left [   \frac{1-p^{-(2s+n)}} {1-p^{-(2s+n+1)}}  \right] .
   \end{align*}
Moreover, for  $\sigma \geq -n/2+\epsilon$, one has
\begin{align*}
    \prod_{ p | c, \, p \, \neq \, 2} D_p(c,s) \left [   \frac{1-p^{-(2s+n)}} {1-p^{-(2s+n+1)}}  \right] \ll |c|^{n/2+\epsilon}
\end{align*}
and 
    \begin{align*}
    D_2(c,s) \, \frac{1-2^{-(2s+n)}} {1-2^{-(2s+n+1)}} \ll_{n, \epsilon, J} 1. 
    \end{align*}
\end{thm}

\begin{proof}

We recall Corollary \ref{Cor 3.4.2}  and get
\begin{align*}
     S(p^{k},c)
            &=  p^{n k -1}(p-1)\sum_{l=1}^{\lfloor k/2 \rfloor} p^{(2-n)l}\,   \, \mathds{1} _{p^{k-2l}|c} + p^{nk}\,  \mathds{1} _{p^k|c} \\
            &\buildrel \rm p \nmid c \over
                      = (p-1)p^{nk/2+k-1} \, \mathds{1} _{2|k,\, k \geq 1}+ \, \mathds{1} _{k=0},
            \end{align*}  
which implies
\begin{align*}
\prod_{  p \nmid c, \, p \, \neq \, 2} D_p(c,s)&=\prod_{p \nmid c, \, p \, \neq \, 2} \sum_{k=0}^{\infty}\frac{S(p^k,c)}{p^{k(s+n+1)}}=\prod_{ p \nmid c, \, p \, \neq \, 2} \left[1+ \sum_{k=1}^{\infty}\frac{S(p^k,c)}{p^{k(s+n+1)}} \right] \\
&= \prod_{ p \nmid c, \, p \, \neq \, 2} \left[1+ \sum_{l=1}^{\infty}\frac{(p-1)p^{-1} }{p^{l(2s+n)}} \right]=\prod_{  p \nmid c, \, p \, \neq \, 2} \left[ \frac{1-p^{-(1+n+2s)}}{1-p^{-(2s+n)}} \right]\\
& = \frac{\zeta(2s+n)}{\zeta(2s+n+1)} \,   \, \prod_{ p | c, \, p \, \neq \, 2} \left [   \frac{1-p^{-(2s+n)}} {1-p^{-(2s+n+1)}}  \right] \, \frac{1-2^{-(2s+n)}} {1-2^{-(2s+n+1)}}.
\end{align*}
Then, we have
\begin{align*}
    D(c,s)&=\prod_p D_p(c,s) =D_2(c,s) \prod_{ p \nmid c, \, p \, \neq \, 2} D_p(c,s) \, \prod_{ p | c, \, p \, \neq \, 2} D_p(c,s)\\
    &=D_2(c,s) \, \frac{1-2^{-(2s+n)}} {1-2^{-(2s+n+1)}} \, \frac{\zeta(2s+n)}{\zeta(2s+n+1)} \, \prod_{ p | c, \, p \, \neq \, 2} D_p(c,s) \left [   \frac{1-p^{-(2s+n)}} {1-p^{-(2s+n+1)}}  \right] .
\end{align*}
The other parts of the theorem follows by Lemmas \ref{Lem 4.2.1} and \ref{Lem 4.2.2} below. 
\end{proof}

\begin{customlem}{4.2.1}
\label{Lem 4.2.1}
\textit{
If $c \, \neq \, 0$, $c^t J^{-1} c =0$, and  $\sigma \geq -n/2+\epsilon$, one has
\begin{align*}
    \prod_{ p | c, \, p \, \neq \, 2} D_p(c,s) \left [   \frac{1-p^{-(2s+n)}} {1-p^{-(2s+n+1)}}  \right] \ll |c|^{n/2+\epsilon}.
\end{align*}
 }
\end{customlem}
\begin{proof}
It suffices to consider $p^{\alpha}\,||\,c$ for some positive integer $\alpha$. We make use of Corollary \ref{Cor 3.4.2} again to get
\begin{align*}
    D_p(c,s) &= \sum_{k=0}^{\infty} \frac{S(p^k,c)}{p^{k(s+n+1)}}= \sum_{k=0}^{\alpha} \frac{1}{p^{k(s+1)}}+\sum_{k=2+\alpha }^{\infty} \frac{p^{nk-1}(p-1)}{p^{k(n+s+1)}} \sum_{l=\lceil (k-\alpha)/2 \rceil }^{\lfloor k/2 \rfloor}
      p^{(2-n)l},
\end{align*}
where the double summation having inner geometric sum can be simplified to
\begin{align*}
          \frac{1-p^{-1}}{1-p^{(2-n)}} \sum_{k=2+\alpha }^{\infty} \frac{(p^{(2-n)\lceil (k -\alpha)/2 \rceil} - p^{(2-n)(\lfloor k/2 \rfloor+1)})}{p^{k(s+1)}}.
               \end{align*}  
Moreover, we observe that            
\begin{align*}
      \sum_{k=2+\alpha }^{\infty} \frac{p^{(2-n)\lceil (k -\alpha)/2 \rceil} }{p^{k(s+1)}} 
           = \frac{p^{-\alpha(s+1)-2s-n}+p^{-(1+\alpha)(s+1)+2-2n-2s}}{1-p^{-(2s+n)}},
      \end{align*}
  and 
  \begin{align*}
      \sum_{k=2+\alpha }^{\infty} \frac{ p^{(2-n)(\lfloor k/2 \rfloor+1)})}{p^{k(s+1)}} 
            = \frac{p^{-(2s+n)\lceil (2+\alpha)/2 \rceil+2-n}+p^{-(2s+n)\lceil (1+\alpha)/2 \rceil+1-n-s}}{1-p^{-(2s+n)}}. 
  \end{align*}
  Then, we get
  \begin{align*}
    \prod_{p | c, \, p \, \neq \, 2} D_p(c,s) \,
       =\prod_{p^\alpha \,||\, c, \, p \, \neq \, 2 } \, \Big [  \sum_{k=0}^\alpha \frac{1}{p^{k(s+1)}}+ \frac{1-p^{-1}}{1-p^{(2-n)}}  \frac{X_2} {1-p^{-(2s+n)}}  \Big ], \tag{4.2} \label{eq4.2}
     \end{align*}
where
\begin{align*}
    X_2=&p^{-\alpha(s+1)-2s-n}+p^{-(1+\alpha)(s+1)+2-2n-2s}+\\&-p^{-(2s+n)\lceil (2+\alpha)/2 \rceil+2-n}-p^{-(2s+n)\lceil (1+\alpha)/2 \rceil+1-n-s}.
\end{align*}
Thus, in the half-plane $\sigma \geq -n/2+\epsilon$, we have
\begin{align*}
&\prod_{ p | c, \, p \, \neq \, 2} D_p(c,s) \left [   \frac{1-p^{-(2s+n)}} {1-p^{-(2s+n+1)}}  \right]\\
        &=\prod_{p^\alpha \,||\, c, \, p \, \neq \, 2 } \, \Big [  \frac{1-p^{-(2s+n)}}{1-p^{-(2s+n+1)}} \sum_{k=0}^\alpha \frac{1}{p^{k(s+1)}}+ \frac{1-p^{-1}}{1-p^{(2-n)}}  \frac{X_2} {1-p^{-(2s+n+1)}}  \Big ]\\
    &\ll \prod_{p^\alpha \,||\, c, \, p \, \neq \, 2 } \, \Big [   \sum_{k=0}^\alpha p^{k(n/2+\epsilon)}+ p^{\alpha(n/2-\epsilon)}+p^{\alpha(n/2+\epsilon)} +p^{\alpha(-n/2+2-\epsilon)}+p^{\alpha(-n+3-\epsilon)}  \Big ].
\end{align*} 
By taking the product over all $p$ dividing $c$, then expanding terms, and noting $\prod_{p^{\alpha}\,||\, c}p^{\alpha} < |c|$, we obtain
\begin{align*}
    \prod_{ p | c, \, p \, \neq \, 2} D_p(c,s) \left [   \frac{1-p^{-(2s+n)}} {1-p^{-(2s+n+1)}}  \right] \ll P_1(|c|) \ll |c|^{n/2+\epsilon},
\end{align*}
 where $P_1(x)$ is a sum of power functions with leading power $n/2+\epsilon$.
\end{proof}
 \begin{customlem}{4.2.2}
 \label{Lem 4.2.2}
 \textit{
 If $c \, \neq \, 0$, $c^t J^{-1} c =0$, and  $\sigma \geq -n/2+\epsilon$, one has
\begin{align*}
    D_2(c,s) \, \frac{1-2^{-(2s+n)}} {1-2^{-(2s+n+1)}} \ll_{n, \epsilon, J} 1.
    \end{align*} }
 \end{customlem} 
 \begin{proof}
 We recall Corollary \ref{Cor 3.5.2}, which gives
\begin{align*}
    S(2^{k},c)= &\sum_{j=0}^{k-2}  2^{j(n/2-1)+k(n/2+1)-1} d_{n,J} \,  \mathds{1} _{2^j|c}\, \mathds{1} _{2|k-j} \,  + 2^{nk} \, \mathds{1} _{2^{k}|c},
\end{align*}
where $d_{n,J}$ is defined in (\ref{eq3.8}). We consider two cases: $2 \nmid c$ and $2 |c$.
\begin{itemize}
    \item If $2 \nmid c$, we have
    \begin{align*}
    S(2^{k},c)
        &=  2^{nk/2} \,  \mathds{1} _{2|k} \, d_{n,J} \, 2 ^{k -1} \mbox { for } k \geq 2, \end{align*}
and so
    \begin{align*}
        D_2(c,s)=\sum_{k=0}^{\infty} \frac{S(2^{k},c)}{2^{k(s+n+1)}}
             = \frac{1-2^{-(2s+n)}+ d_{n,J} \, 2 ^{-(2s+n+1)}}{1-2^{-(2s+n)}}.
    \end{align*}
   Thus, for $\sigma  \geq -n/2 + \epsilon$, we get
    \begin{align*}
             D_2(c,s) \frac{1-2^{-(2s+n)}}{1-2^{-(2s+n+1)}}
              &= \frac{1-2^{-(2s+n)}+ d_{n, J} \, 2 ^{-(2s+n+1)}}{1-2^{-(2s+n+1)}} \\
        &\ll \frac{1-2^{-2\epsilon}+ d_{n, J} \, 2 ^{-(1+2\epsilon)}}{1-2^{-(1+2\epsilon)}} \ll_{n, \epsilon, J} 1. 
           \end{align*}
    \item If $2 | c$, i.e., $2^\alpha \,||\, c$, for some positive integer $\alpha$,  we have
    \begin{align*}
        D_2(c,s)=\sum_{k=0}^{\infty} \frac{S(2^{k},c)}{2^{k(s+n+1)}}  
        =\sum_{k=2}^{\alpha+2}\sum_{j=0}^{k-2}  \frac{2^{n(k+j)/2}}{2^{k(s+n+1)}} \,  \mathds{1} _{2|k-j} \, d_{n, J} \, 2 ^{k -j -1}+\sum_{k=0}^{\alpha} \frac{1}{2^{k(s+1)}},
    \end{align*}
    where the double summation is equal to
    \begin{align*}
              \frac{d_{n, J} \, }{2^{n-1}-2} \left (\sum_{k=2}^{\alpha+2} \frac{1}{2^{k(s+1)}} - \sum_{k=2}^{\alpha+2} \frac{2^{(2-n)\lfloor k/2 \rfloor}}{2^{k(s+1)}}  \right ).
    \end{align*}
    We also observe that 
    \begin{align*}
        \sum_{k=2}^{\alpha+2} \frac{2^{(2-n)\lfloor k/2 \rfloor}}{2^{k(s+1)}} &= \sum_{l=1}^{\lfloor (\alpha+2)/2 \rfloor} \frac{2^{(2-n)l}}{2^{2l(s+1)}} +\sum_{l=1}^{\lfloor (\alpha+1)/2 \rfloor} \frac{2^{(2-n)l}}{2^{(2l+1)(s+1)}} \\
        &= \sum_{l=1}^{\lfloor (\alpha+2)/2 \rfloor} \frac{1}{2^{l(2s+n)}} + \sum_{l=1}^{\lfloor (\alpha+1)/2 \rfloor} \frac{2^{-(s+1)}}{2^{l(2s+n)}}
         \leq \frac{1+2^{-(s+1)}}{1-2^{-(2s+n)}}. 
    \end{align*}
    \end{itemize}
          Hence,
    \begin{align*}
        &D_2(c,s) \frac{1-2^{-(2s+n)}}{1-2^{-(2s+n+1)}} \leq \\
        & \leq \frac{d_{n,J} }{2^{n-1}-2} \left (\frac{1+2^{n/2-1-\epsilon}}{1-2^{-\epsilon}} + 1+2^{n/2-1-\epsilon}  \right )\frac{1}{1-2^{-(2\epsilon+1)}}+\frac{1+2^{n/2-1-\epsilon}}{(1-2^{-\epsilon})(1-2^{-(2\epsilon+1)})} \\
        &\ll_{n, \epsilon, J } 1.
        \end{align*}
\end{proof}
{\centering 
\subsection*{4.3}{The case $c \, \neq \, 0$ and $c^t J^{-1} c \, \neq \, 0$}
\par }

Lastly, 
it is also interesting to see that $D(c,s)$, where $c \, \neq \, 0$ and $c^t J^{-1} c \, \neq \, 0$, can be written in terms of a quotient of Dirichlet L-functions, $\frac{L(s+n/2+1/2,\,\kappa_{c,J})}{L(2s+n+1,\, \kappa_{c,J})}$, of character $\kappa_{c,J} $ defined as in Theorem \ref{Thm 4.3}. Moreover, the character $\kappa_{c,J} $ turns out to be a product of Dirichlet characters, which will be discussed in details in Lemma \ref{Lem 6.3} of Section 6.
\begin{thm}
\label{Thm 4.3}
If $c \, \neq \, 0$ and $ c^t J^{-1} c \, \neq \, 0$, one has
\begin{align*}
    D(c,s)   = \frac{L(s+n/2+1/2,\kappa_{c,J})}{L(2s+n+1,\kappa_{c,J})} \frac{D_2(c,s)}{\rho_{c,J}(2)} \prod_{p|c, \, p \, \neq \, 2} \frac{D_p(c,s)}{\rho_{c,J}(p)} \prod_{p\nmid c, \, p| c^t J^{-1} c , \, p \, \neq \, 2} \frac{D_p(c,s)}{\rho_{c,J}(p)} ,
\end{align*}
where 
\begin{align*}
    \kappa_{c,J} (p)&=i^{(n-1)( p-1)^2/4}\,\left(\frac{c^t J^{-1} c\, \det J}{p} \right )_K, \mbox { and } \rho_{c,J}(p)= 1+ \frac{\kappa_{c,J} (p) }{p^{s+n/2+1/2}}, \\ \mbox { with } \left(\frac{c^t J^{-1} c\, \det J}{\cdot} \right)_K &\mbox { is the Kronecker symbol. }.
\end{align*}
Moreover, for $\sigma \geq -n/2+\epsilon$, one has
\begin{align*}
    \prod_{p|c, \, p \, \neq \, 2} \frac{D_p(c,s)}{\rho_{c,J}(p)} &\leq  |c|^{3n/2-7/2},\\
    \prod_{p\nmid c, \, p| c^t J^{-1} c , \, p \, \neq \, 2} \frac{D_p(c,s)}{\rho_{c,J}(p)} &\leq |c|^{1/2+\epsilon},\\
    \mbox { and }
     \frac{D_2(c,s)}{\rho_{c,J}(2)}  &\ll_{n, \epsilon} 1      .
     \end{align*}
\end{thm}
\begin{proof}
By applying Theorem \ref{Thm 3.4} with $p \nmid c$  and $p\nmid c^t J^{-1} c$, we have
\begin{align*}
    D_p(c,s)&=\sum_{k=0}^{\infty} \frac{S(p^k,c)}{p^{k(s+n+1)}} = 1+ \frac{ i^{(n-1)( p-1)^2/4}}{p^{n/2+s+1/2}} \,\left(\frac{ c^t J^{-1} c\, \det J}{p} \right )_K . 
\end{align*}
By setting $\kappa_{c,J} (p):=i^{(n-1)( p-1)^2/4}\,\left(\frac{ c^t J^{-1} c \, \det J}{p} \right )_K $ and $\rho_{c,J}(p):= 1+ \frac{\kappa_{c,J} (p) }{p^{s+n/2+1/2}}$ , we get
\begin{align*}
    \prod_{p\nmid c, \, p \nmid c^t J^{-1} c,  \, p \, \neq \, 2} D_p(c,s)  &=\prod_{p \nmid c, \, p\nmid c^t J^{-1} c, \, p \, \neq \, 2 } \, \rho_{c,J}(p)\\
    &= \frac{1}{\rho_{c,J}(2)} \, \prod_{p} \rho_{c,J}(p)   \prod_{p|c, \, p \, \neq \, 2} \frac{1}{\rho_{c,J}(p)} \, \prod_{p\nmid c, \, p| c^t J^{-1} c , \, p \, \neq \, 2} \frac{1}{\rho_{c,J}(p)} \, \\ 
    &= \frac{1}{\rho_{c,J}(2)} \, \frac{L(s+n/2+1/2,\kappa_{c,J})}{L(2s+n+1,\kappa_{c,J})}  \prod_{p|c, \, p \, \neq \, 2} \frac{1}{\rho_{c,J}(p)} \, \prod_{p\nmid c, \, p| c^t J^{-1} c , \, p \, \neq \, 2} \frac{1}{\rho_{c,J}(p)} ,
\end{align*}
which yields
\begin{align*}
    D(c,s)
    &= D_2(c,s) \prod_{p|c, \, p \, \neq \, 2} D_p(c,s) \prod_{p\nmid c, \, p| c^t J^{-1} c , \, p \, \neq \, 2} D_p(c,s) \prod_{p\nmid c, \, p \nmid c^t J^{-1} c,  \, p \, \neq \, 2} D_p(c,s)\\
    &=\frac{D_2(c,s)}{\rho_{c,J}(2)}\,\frac{L(s+n/2+1/2,\kappa_{c,J})}{L(2s+n+1,\kappa_{c,J})}  \prod_{p|c, \, p \, \neq \, 2} \frac{D_p(c,s)}{\rho_{c,J}(p)} \prod_{p\nmid c, \, p| c^t J^{-1} c , \, p \, \neq \, 2} \frac{D_p(c,s)}{\rho_{c,J}(p)}. 
\end{align*}
Hence, the remaining parts of the theorem are justified by the following Lemmas \ref{Lem 4.3.1}, \ref{Lem 4.3.2}, and \ref{Lem 4.3.3}.
\end{proof}
\begin{customlem}{4.3.1}
\label{Lem 4.3.1}
\textit{
If $c \, \neq \, 0$, $c^t J^{-1} c \, \neq \, 0$, and  $\sigma \geq -n/2+\epsilon$, one has
\begin{align*}
    \prod_{p|c, \, p \, \neq \, 2} \frac{D_p(c,s)}{\rho_{c,J}(p)} \ll |c|^{3n/2-7/2}.
\end{align*}
}
\end{customlem}
\begin{proof}
Since $p|c$ implies $p^2|c^t J^{-1} c$, we consider $p^\alpha \,||\, c$ and $p^{2\alpha+\beta} \,||\, c^t J^{-1} c$ for some integers $\alpha \geq 1$ and $\beta \geq 0$. Applying Theorem \ref{Thm 3.4} to $p^\alpha \,||\, c$ and $p^{2\alpha+\beta} \,||\, c^t J^{-1} c$, we get  
\begin{align*}
    D_p(c,s)&=\sum_{k=0}^{\infty} \frac{S(p^k,c)}{p^{k(s+n+1)}} = X_3+X_4+X_5+X_6,
\end{align*}
where
\begin{align*}
    X_3
    &= \sum_{k=\alpha+\beta+1}^{2\alpha+\beta+1} \frac{p^{n(\alpha+\beta/2+1/2)-2\alpha-\beta-3/2}\,i^{(n-1)( p-1)^2/4} \,\left(\frac{ \det J}{p} \right )_K }{p^{k(n+s-1)}}  \left (\frac{c^t J^{-1} c /p^{2\alpha+\beta}}{p} \right )_K \, \mathds{1} _{2|\beta} ,\\
    X_4&=\sum_{k=0}^\alpha \frac{1}{p^{ k(s+1)}}, \quad
    X_5=\sum_{k=\alpha+\beta+1}^{2\alpha+\beta+1} \frac{-p^{n(1/2+\alpha+\beta/2)-2-2\alpha-\beta}  }{p^{k(n+s-1)}}   \, \mathds{1} _{2\nmid \beta}, \\ \mbox{ and } X_6&= \sum_{k=\mbox{\tiny{max}}\{\lceil 1+\alpha+\beta/2 \rceil,\alpha+2\}}^{2\alpha+\beta} \frac{p^{ -1}(p-1)}{p^{k(s+1)}}   \sum_{l=\mbox{\tiny{max}}\{\lceil k-\alpha-\beta/2 \rceil,\lceil\frac{k-\alpha}{2}\rceil \}}^{\lfloor \frac{k}{2}\rfloor}  p^{(2-n)l}. 
\end{align*}
For $\sigma \geq -n/2+\epsilon$, we observe that 
    \begin{align*}
    X_3
    &\ll \sum_{k=\alpha+\beta+1}^{2\alpha+\beta+1} \frac{p^{(2\alpha+\beta)(n-5/2)}}{p^{k(n/2-1+\epsilon)}}  , \quad 
    X_4\ll\sum_{k=0}^\alpha p^{k(n/2-1-\epsilon)}, \\
    X_5
    &\ll \sum_{k=\alpha+\beta+1}^{2\alpha+\beta+1} \frac{p^{(2\alpha+\beta)(n-3)}}{p^{k(n/2-1+\epsilon)}}  , \mbox{ and } X_6= \sum_{k=\mbox{\tiny{max}}\{\lceil 1+\alpha+\beta/2 \rceil,\alpha+2\}}^{2\alpha+\beta}p^{k(n/2-1-\epsilon)} . 
\end{align*}
Moreover,
\begin{align*}
    \rho_{c,J}(p)^{-1} = \left ( 1+ \frac{\chi_c \eta_c}{p^{s+n/2+1/2}}\right )^{-1} 
    &\ll \left ( 1- \frac{1}{p^{\epsilon+1/2}}\right )^{-1}  = \frac{p^{\epsilon+1/2}}{p^{\epsilon+1/2} -1}, \\ \mbox { and } (p^{\epsilon+1/2} -1)^{-1} &< 1 \mbox { for } p \geq 5.
\end{align*}              
Hence, 
\begin{align*}
    \prod_{p|c, \, p \, \neq \, 2} \frac{D_p(c,s)}{\rho_{c,J}(p)} \ll &\prod_{p|c, \, p \, \neq \, 2} p^{\epsilon+1/2}\Big [ \sum_{k=\alpha+\beta+1}^{2\alpha+\beta+1} \frac{p^{(2\alpha+\beta)(n-5/2)}}{p^{k(n/2-1+\epsilon)}} \\ &+ \sum_{k=0}^\alpha p^{k(n/2-1-\epsilon)} + \sum_{k=\mbox{\tiny{max}}\{\lceil 1+\alpha+\beta/2 \rceil,\alpha+2\}}^{2\alpha+\beta}p^{k(n/2-1-\epsilon)} \Big ],
\end{align*}
which is bounded by a sum of power functions of $|c|$ with leading power $3n/2-7/2$. 
\end{proof}
\begin{customlem}{4.3.2}
\label{Lem 4.3.2}
\textit{
If $c \, \neq \, 0$, $c^t J^{-1} c \, \neq \, 0$, and  $\sigma \geq -n/2+\epsilon$, one has
\begin{align*}
    \prod_{p\nmid c, \, p| c^t J^{-1} c , \, p \, \neq \, 2} \frac{D_p(c,s)}{\rho_{c,J}(p)} \ll 
    |c|^{1/2+\epsilon}. 
    \end{align*}
}
\end{customlem}
\begin{proof}
Since $p \nmid c$ and $p|c^t J^{-1} c$, it suffices to consider $p^\alpha\,||\,c^t J^{-1} c$ for some positive integer $\alpha$. Given $p^\alpha\,||\,c^t J^{-1} c$, Theorem \ref{Thm 3.4} yields 
\begin{align*}
    D_p(c,s)
    &= 1+ \frac{ i^{(n-1)( p-1)^2/4}}{p^{n/2+s+1/2+\alpha n/2+\alpha s}} \,\left(\frac{ \det J}{p} \right ) \, \left(\frac{c^t J^{-1} c/p^\alpha}{p} \right ) \,  \, \mathds{1} _{2|\alpha}+ \\
    &\quad + \frac{\, \mathds{1} _{2 \nmid \alpha}}{p^{n/2+s+1+\alpha n/2+\alpha s}}+\sum_{k=\lceil 1+\alpha/2 \rceil }^\alpha \frac{ (1-p^{-1})}{p^{k(n/2+s)}} \,   \, \mathds{1} _{2 | k}. \tag{4.3} \label{eq4.3}
\end{align*}
We observe that the two middle terms in (\ref{eq4.3}) are bounded by $p^{\alpha(-\epsilon)}$ while the last one is bounded by $\sum_{l=\lceil \frac{\lceil 1+\alpha/2 \rceil}{2} \rceil }^{\lfloor \alpha/2 \rfloor} p^{l(-2\epsilon)}$. Thus, we obtain
\begin{align*}
    \prod_{p\nmid c, \, p| c^t J^{-1} c , \, p \, \neq \, 2} \frac{D_p(c,s)}{\rho_{c,J}(p)} \ll &\prod_{p|c, \, p \, \neq \, 2} p^{\epsilon+1/2}\Big [ 1+ p^{\alpha(-\epsilon)} + \sum_{l=\lceil \frac{\lceil 1+\alpha/2 \rceil}{2} \rceil }^{\lfloor \alpha/2 \rfloor} p^{l(-2\epsilon)}\Big ],
\end{align*}
which is bounded by a sum of power functions of $|c|$ with leading power $1/2+\epsilon$.
\end{proof}
\begin{customlem}{4.3.3}
\label{Lem 4.3.3} 
\textit{
If $c \, \neq \, 0$, $c^t J^{-1} c \, \neq \, 0$, and  $\sigma \geq -n/2+\epsilon$, one has 
\begin{align*}
     \frac{D_2(c,s)}{\rho_{c,J}(2)}  \ll_{n, \epsilon} 1.
\end{align*} }
\end{customlem}
\begin{proof}

We recall Theorem \ref{Thm 3.5} to get
\begin{align*}
    S(2^{k},c)= &\sum_{j=0}^{k-2}  2^{j(n/2-1)+k(n/2+1)-1} d_{n,j,k,J} \,  \mathds{1} _{2^j|c}\, \mathds{1} _{2|k-j} \,    +\\&+ \sum_{j=0}^{k-2}  2^{j(n/2-1)+k(n/2+1)-3/2} \, \left (\frac{2}{\det J} \right )_L\, e_{n,j,J} \,  \mathds{1} _{2^j|c} \, \mathds{1} _{2\nmid k-j} \, \mathds{1} _{2^{k-j-3} \,||\, c_{j,J} } + \\ &+2^{nk} \, \mathds{1} _{2^{k}|c}, \tag{4.4} \label{eq4.4}
\end{align*}
and we examine 
\begin{align*}
    D_2(c,s)=\sum_{k=0}^{\infty} \frac{S(2^{k},c)}{2^{k(n+s+1)}}
\end{align*}
by computing contribution of $S(2^{k},c)$ (\ref{eq4.4}) to $D_2(c,s)$. For the first summation of $S(2^{k},c)$ in (\ref{eq4.4}), we proceed by considering the divisibility by $2$ of $c$.
    \begin{itemize}
        \item If $2 \nmid c$, then $j=0$ and 
       it equals to  $2^{k(n/2+1)-1} \,d_{n,0,J}\,  \mathds{1} _{2|k, \, k \geq 2}$. We further consider $2 \nmid c^t J^{-1} c$ to get its contribution to $D_2(c,s)$ as
       \begin{align*} 
       \frac{-1}{2^{2s+1+n/2}}  \, \chi_4(c^t J^{-1} c) \,  \sin \left(\frac{(2m-n) \pi}{4}\right) \, \mathds{1} _{c^t J^{-1} c\,=\,\pm1} \ll \frac{1}{2^{-n/2+1+2 \epsilon}} \ll_{n, \epsilon} 1.
       \end{align*}
       Otherwise, we assume $2^{\alpha} \,||\, c^t J^{-1} c$, where $\alpha \in \mathbb{Z}_{\geq 1}$. Then, the contribution equals to
        \begin{align*}
           &\sum_{l=1}^{\lfloor \alpha/2 \rfloor}  \frac{2^{n/2-1}}{2^{l(2s+n)}} \cos \left(\frac{(2m-n) \pi}{4}\right) - \frac{2^{n(\alpha+1)/2} \,  \mathds{1} _{2|(\alpha+1)}  2^{\alpha}}{2^{(\alpha+1)(n+s+1)}}2^{n/2} \cos \left(\frac{(2m-n) \pi}{4}\right)
           +\\&\quad - \frac{2^{n(\alpha+2)/2} \,  \mathds{1} _{2|(\alpha+2)}  2^{\alpha+1}}{2^{(\alpha+2)(n+s+1)}} \, \chi_4(c^t J^{-1} c/2^{\alpha}) \,2^{n/2} \sin \left(\frac{(2m-n) \pi}{4}\right) \ll
           \\&\ll \frac{2^{n/2-1}}{1-2^{-2 \epsilon}}+ 2^{n/2-\epsilon(\alpha+1)-1} \ll_{n, \epsilon, \alpha} 1.
       \end{align*}
               \item If $2|c$, we assume $2^{\alpha} \,||\,c$ and $2^{2 \alpha +\beta} \,||\, c^t J^{-1} c$, for some integers $\alpha \geq 1, \beta \geq 0.$ Then, the condition $2^j|c$ makes the contribution of the first summation of $S(2^{k},c)$ in (\ref{eq4.4}) to $D(c,s)$ equal to
        \begin{align*}
           \sum_{k=2}^{\alpha+2} \sum_{j=0}^{k-2} \frac{  2^{j(n/2-1)+k(n/2+1)-1} d_{n,j,k,J} \,  \mathds{1} _{2|k-j}}{2^{k(n+s+1)}}, \tag{4.5} \label{eq4.5}
       \end{align*}
    By (\ref{eq3.8}), $d_{n,j,k,J} \leq 2^{(n-1)/2}$ and so the contribution to $D(c,s)$ (\ref{eq4.5}) is bounded by
    \begin{align*}
        \frac{(\alpha+1) 2^{-n/2-7/2}}{1-2^{-\epsilon}} \ll_{n, \epsilon, \alpha} 1.
    \end{align*}
    \end{itemize}
    
            Similarly, we return to $S(2^{k},c)$ in (\ref{eq4.4}) and consider its second summation under two cases, $2 \nmid c$ or $2|c$. By noting that $e_{n,0,J} \leq 2^{(n+1)/2}$ as seen in (\ref{eq3.7}), the former case yields a contribution of 
         \begin{align*}
        \sum_{l=1}^{\infty} \,  e_{n,0,J}\, 2^{-3/2-n/2-s}\frac{1}  {2^{l(2s+n)}} 
           &\ll \frac{2^{(n+1)/2}}{1-2^{-2\epsilon}} 2^{-3/2-3\epsilon} \ll_{n, \epsilon } 1,
    \end{align*}
    while the latter one with $2^{\alpha} \,||\,c$ and $2^{2 \alpha +\beta} \,||\, c^t J^{-1} c$, for some integers $\alpha \geq 1, \beta \geq 0,$ gives us
    \begin{align*}
    \sum_{k=2}^{\alpha+2}\sum_{j=0}^{k-2}  \frac{2^{n(k+j)/2}}{2^{k(s+n+1)}} \,  \mathds{1} _{2\nmid k-j} \, e_{n,j,J}\, 2 ^{k -j -3/2} \, \mathds{1} _{2^{k-j-3}\,||\,c_{2,j,J}}=0  ,
         \end{align*}
         by observing that  $e_{n,j,J} \, \neq \, 0$ if and only if $2^{k-j-3}\,||\,c_{2,j,J}$, which yields $k+j-3=2\alpha+\beta,$ and so $j=2\alpha+\beta-k+3$. Also, the conditions $0\leq j \leq k-2$ and $2 \leq k \leq  \alpha +2$ imply $\beta \leq -1$, which is a contradiction.
         
         We complete checking the boundedness of $D_2(c,s)$ by observing that the infinite sum $\sum_{k=0}^{\infty} \frac{\, \mathds{1} _{2^k|c}}{2^{k(s+1)}}$, which comes from the last term of $S(2^{k},c)$ in (\ref{eq4.4}), is identically to $1$ if $2 \nmid c$. Otherwise, with $2^{\alpha} \,||\,c$, for some $\alpha \in \mathbb{Z}_{\geq 1}$, we have
         \begin{align*}
        \sum_{k=0}^{\infty} \frac{\, \mathds{1} _{2^k|c}}{2^{k(s+1)}}=\sum_{k=0}^{\alpha} \frac{1}{2^{k(s+1)}} \ll \sum_{k=0}^{\alpha} 2^{k(n/2-1-\epsilon)} \ll_{\epsilon, n, \alpha} 1.
    \end{align*}
    
Lastly, we prove the theorem by considering $\rho_{c,J}(p) $ and observing that
\begin{align*}
    \rho_{c,J}(p)^{-1} 
     \ll\left ( 1- \frac{1}{p^{\epsilon+1/2}}\right )^{-1} 
    \leq \frac{2^{\epsilon+1/2}}{2^{\epsilon+1/2} -1} \ll_{\epsilon} 1. 
\end{align*}
\end{proof}

\section{The Archimedean factor }
\label{Sec 5}
We recall that the Archimedean factor $g(c,s)$ is the Mellin transform in the second variable applied to the Fourier transform in the first variable of two-variable smooth function. That is, it is defined by
\begin{align*}
    g(c,s)=  \int_{0}^\infty \int_{\mathbb{R}^n} \omega(x)\,h\left(\frac{F(x)}{y},y\right)e\left(-\frac{c\cdot x}{y}\right)  \,dx \, y^s \, \frac{dy}{y}.
\end{align*}
For the purpose of this paper, considering a complement case in the work of Getz, we subsequently apply
Theorems 4.1 and 4.2 in \cite{G} to the smallest and the most basic number field, the field $\mathbb{Q}$ of rational numbers, and obtain the following propositions. 
\begin{prop}
\label{Prop 5.1}
    Let $s=\sigma + it$, $c \, \neq \, 0$, and $N$ be any positive integer. For  $N>  \sigma_1 > \sigma > \sigma_2 > -n/2$, $g(c,s)$ converges absolutely, and      \begin{align*}
     \qquad \qquad  \qquad  \qquad  g(c,s) &\ll_{\omega, \Phi} \frac{|c|^{-N}}{\sqrt{\sigma_1^2 +t^2+1}} \left ( \frac{1}{\sigma_2+n/2} - \frac{1}{\sigma_1-N} \right) .  \qquad  \qquad \qquad \qquad \square 
    \end{align*}
\end{prop}
\begin{prop}
\label{Prop 5.2} If $c = 0$, $g(0,s)$ converges absolutely for $\sigma > -3/4$. In the strip $N> \sigma_1 > \sigma > -3/4$,
\begin{align*}
    g(0,s) \ll_{\omega, \Phi} \frac{1}{\sqrt{\sigma_1^2 +t^2+1}} \left ( \frac{1}{-3/4+n/2} - \frac{1}{\sigma_1-N} \right). 
\end{align*}
Moreover, $g(0,s)$ can be meromorphically continued to the half plane $\sigma > -n/2$ by the virtue of
$$g(0,s) =\pi^{-s-1/2} \, \frac{\Gamma\left(\frac{s+1}{2}\right)}{\Gamma\left(\frac{-s}{2}\right)}\int_{\mathbb{R}} Z( z,-s) \, dz,$$
where 
\begin{align*}
    Z( z,-s) = \int_0^{\infty} \int_{\mathbb{R}^n} \int_{\mathbb{R}} \int_{\mathbb{R}} \omega(x)\, h(r,v)\, e(r z v) \, dr \,e(-vy)\, dv \, e(-z F(x) ) \, dx\, y^{-s} \,\frac{dy}{y}, \tag{5.1} \label{eq5.1}
\end{align*}
In particular, in the strip $N> -1/4\geq \sigma_1 > \sigma > \sigma_2 > -n/2$,
\begin{align*}
   \qquad \quad  \qquad  \qquad  \int_{\mathbb{R}} Z(z,-s) \, dz \ll \frac{1}{\sqrt{1/16 +t^2}} \left ( \frac{1}{\sigma_2+n/2} + \frac{1}{1/4+N} \right). \qquad  \qquad \qquad \quad \square
\end{align*}
\end{prop}
We end this section by making some remarks of these above results. The bound of $g(c,s)$ in Proposition \ref{Prop 5.1} implies that $g(c,s)$ is rapidly decreasing as $|c|$ approaches to infinity or $|t|$ goes to infinity. Furthermore, having an explicit form of the meromorphic continuation of $g(0,s)$ beyond the line $
\sigma = -3/4$ in Proposition \ref{Prop 5.2}, we can deduce the poles and zeros of $g(0,s)$, which ultimately contribute to the poles of $D(0,s)\, g(0,s)$ that will be discussed in the next section. In fact, the meromorphic continuation of gamma function produces poles of $g(0,s)$ at $s=1-2k$, where $ 1 \leq k < \lfloor \frac{n+2}{4} \rfloor$, providing that $\int_{\mathbb{R}} Z({z},-s) \; dz \, \neq \, 0$ at these poles, and zeros of $g(0,s)$ at positive even integers and at the zeros of $\int_{\mathbb{R}} Z({z},-s) \; dz$ that are not poles of $g(0,s)$. 
\section{The main theorem }
\label{Sec 6}
We revisit our zero-counting function defined in Proposition \ref{Prop 2.2},
\begin{align*}
N(B)&= \frac{c_B  B^{n-1}}{2 \pi i} \sum_{c \in \mathbb{Z}^n} \int_{\mathrm{Re}(s)=\sigma_0} D(c,s)B^{s}g(c,s)\,ds. \tag{6.1} \label{eq6.1}
\end{align*} 
We started our study with $\sigma_0 >1$ so that $D(c,s)$ converges absolutely. Then we were able to meromorphically continue both $D(c,s)$ and $g(c,s)$ to the half plane $\sigma > -n/2$, which was previously discussed in Theorems \ref{Thm 4.1}-\ref{Thm 4.2}-\ref{Thm 4.3} and Propositions \ref{Prop 5.1}-\ref{Prop 5.2} in Sections \ref{Sec 4} and \ref{Sec 5}. In light of Cauchy's residue theorem, we want to make a contour shift for the integrand $D(c,s) B^s g(c,s)$ in (\ref{eq6.1}). While shifting, we search for poles of $D(c,s)\,g(c,s)$, viewed as a function of $s$, at different values of $c$, and pick up the corresponding residues, which eventually contributes to our asymptotic evaluation. We present poles of $D(c,s)\,g(c,s)$ in Lemmas \ref{Lem 6.1}, \ref{Lem 6.2}, and \ref{Lem 6.3}, and derive our asymptotic formula in the main theorem, Theorem \ref{Thm 6.4}.

\begin{lem}
\label{Lem 6.1} Let $s=\sigma+it$ and $\epsilon > 0$. 
If $c=0$, then $D(0,s)\,g(0,s)$ has at worst simple poles at $s=0, -1, \mbox { and } \frac{1}{2} -\frac{n}{2}$ in the half plane $\sigma \geq -n/2 +\epsilon$.
\end{lem}
\begin{proof}
Using Theorem \ref{Thm 4.1}, we can analyze poles and zeros of $D(0,s)$ in the half plane $\sigma \geq -n/2 +\epsilon$. Indeed, Theorem \ref{Thm 4.1} yields
\begin{align*}
     D(0,s)&= D_2(0,s) \, \frac{(1-2^{-(s+1)})(1-2^{-(2s+n)})}{1-2^{-(2s+n+1)}} \, \frac{\zeta(s+1)\,\zeta(2s+n)}{\zeta(2s+n+1)}, 
     \end{align*}
where
\begin{align*}
     \overline{D_2}(0,s):&= D_2(0,s) \, \frac{(1-2^{-(s+1)})(1-2^{-(2s+n)})}{1-2^{-(2s+n+1)}} 
          \ll_{n, \epsilon} 1 .
     \end{align*}
Restricted to the regime $\sigma \geq -n/2 +\epsilon$, the analytic continuation of Riemann Zeta function yields simple poles of $D(0,s)$ at $s=0$ and $s=\frac{1}{2} -\frac{n}{2}$. 

We also recognize the zeros of $D(0,s)$ in the half plane $\sigma \geq -n/2 +\epsilon$, which are at $s=-1-2k$, where positive integer $k$ with $\frac{n-3}{4} \, \neq \, k < \lfloor \frac{n-2}{4} \rfloor$, and at non-trivial zeros of $\zeta(s+1)$ and $\zeta(2s+n)$, and at zeros of $\overline{D_2}(0,s)$. Together with the poles and zeros of $g(0,s)$, which are previously remarked in Section \ref{Sec 5}, we obtain the poles of $D(0,s)\,g(0,s)$ as desired. 

As for application, it is vital to see whether  $\overline{D_2}(0,s)$ vanishes at these poles. It turns out that the poles at $s=0, -1, \mbox { and } \frac{1}{2} -\frac{n}{2}$ are not cancelled by zeros of $\overline{D_2}(0,s)$. We defer our justification to Appendix, Proposition \ref{Prop A1.1}.
\end{proof}

\begin{lem}
\label{Lem 6.2} Let $s=\sigma+it$ and $\epsilon > 0$. 
If $c \, \neq \, 0$ and $c^t J^{-1} c = 0$, then
$D(c,s)\,g(c,s)$ has at worst a simple pole at $s=\frac{1}{2} -\frac{n}{2}$ in the half plane $\sigma \geq -n/2 +\epsilon$.
\end{lem}

\begin{proof}

We recall Theorem \ref{Thm 4.2} in order to locate poles and zeros of $D(c,s)$ given the conditions $c \, \neq \, 0$ and $c^t J^{-1} c = 0$. That is,
\begin{align*}
    D(c,s)=D_2(c,s) \, \frac{1-2^{-(2s+n)}} {1-2^{-(2s+n+1)}} \, \frac{\zeta(2s+n)}{\zeta(2s+n+1)} \, \prod_{ p | c, \, p \, \neq \, 2} D_p(c,s) \left [   \frac{1-p^{-(2s+n)}} {1-p^{-(2s+n+1)}}  \right] . 
    \end{align*}
Together with Lemmas \ref{Lem 4.2.1} and \ref{Lem 4.2.2}, the analytic continuation of $\zeta$ function yields a single simple pole of $D(c,s)$ at $s=\frac{1}{2} -\frac{n}{2}$ in the half plane $\sigma \geq -n/2 +\epsilon$, which is also not absorbed by other factors of $D(c,s)$ as seen in Appendix, Propositions \ref{Prop A2.1} and \ref{Prop A2.2}. Nevertheless, the zeros of $D(c,s)$ are non-trivial zeros of $\zeta(2s+n)$ and zeros of other factors in the above expansion of $D(c,s)$. Apllying Proposition \ref{Prop 5.1}, $g(c,s)$ is analytic in the region in question. Hence, the lemma follows.
\end{proof}

\begin{lem}
\label{Lem 6.3} Let $s=\sigma+it$ and $\epsilon > 0$. If $c \, \neq \, 0$ and $c^t J^{-1} c \, \neq \, 0$, then $D(c,s)\,g(c,s)$ has at worst a simple pole at $s=\frac{1}{2} -\frac{n}{2}$ in the half plane $\sigma \geq -n/2 +\epsilon$ whenever  $i^{(n-1)^2/2}\,c^t J^{-1} c\,\det J$ is a perfect square. Moreover, when $n=5$, $m=4$, and $c^t J^{-1} c =1$, where $m = \#\{b_r: b_r=-1, r=1, \ldots, n\}$, this pole is cancelled by a zero of $D(c,s)$, i.e., $D(c,s)\,g(c,s)$ is analytic for $\sigma \geq -n/2 +\epsilon$.
\end{lem}
\begin{proof}
Similar to the previous two lemmas, we use Theorem \ref{Thm 4.3},  when $c \, \neq \, 0$ and $c^t J^{-1} c \, \neq \, 0$, to study poles and zeros of $D(c,s)$. That is,
\begin{align*}
    D(c,s)   = \frac{L(s+n/2+1/2,\kappa_{c,J})}{L(2s+n+1,\kappa_{c,J})} \frac{D_2(c,s)}{\rho_{c,J}(p)} \prod_{p|c, \, p \, \neq \, 2} \frac{D_p(c,s)}{\rho_{c,J}(p)} \prod_{p\nmid c, \, p| c^t J^{-1} c , \, p \, \neq \, 2} \frac{D_p(c,s)}{\rho_{c,J}(p)} ,
\end{align*}
Together with Lemmas \ref{Lem 4.3.1}, \ref{Lem 4.3.2} and \ref{Lem 4.3.3}, the analytic continuation of Dirichlet L-function supplies a simple pole of $D(c,s)$ at $s=\frac{1}{2} -\frac{n}{2}$ in the half plane $\sigma \geq -n/2 +\epsilon$ if and only if $\kappa_{c,J}$ becomes trivial character, namely $\chi_0$. Here the zeros of $D(c,s)$ are non-trivial zeros of $L(s+n/2+1/2,\kappa_{c,J})$ and zeros of other factors in the above expansion of $D(c,s)$. 

In contrast to Lemma \ref{Lem 6.2}, the pole at $s=1/2-n/2$ is cancelled by a zero of $D(c,s)$ only when $n=5$, $m=4$, and $c^t J^{-1} c =1$, where $m = \#\{b_r: b_r=-1, r=1, \ldots, n\}$. The verifying computations were done in Appendix, Propositions \ref{Prop A3.1}, \ref{Prop A3.2}, and \ref{Prop A3.3}. 

It remains to detect such $c$'s that make $\kappa_{c,J}$ trivial. In fact, we recall 
        \begin{align*}
         \kappa_{c,J}(p)=i^{(n-1)( p-1)^2/4}  \, \left(\frac{c^t J^{-1} c \, \det J}{p} \right )_K 
         \end{align*}
        and observe the facts that $i^{\frac{n-1}{4}( p-1)^2}$ and $\left(\frac{ -1}{p} \right )_L$ depend only on $p$ modulo $4$. Using Corollary 3.3 in \cite{AG}, we have 
        \begin{itemize}
            \item If $c^t J^{-1} c \, \det J \not\equiv \, 3 \mbox { (mod } 4 \mbox {)}$, we then define $$\displaystyle{\chi_{c1}(p):=\left(\frac{c^t J^{-1} c \, \det J }{p} \right )_K} \mbox {, i.e., a Dirichlet character modulo } 4|c^t J^{-1} c \, \det J|, \mbox { and } $$
            \begin{align*}
                \eta_{c1}(p):= \left\{ \begin{array}{cll} 
    1 & \mbox{if } p \, \equiv \, 1 \mbox { (mod } 4 \mbox {)}    \\ \\
                         (-1)^{(n-1)/2} & \mbox{if } p \, \equiv \, 3 \mbox { (mod } 4 \mbox {)} \quad \mbox {, i.e., a Dirichlet character modulo } 4. \\ \\
                         0 & \mbox{if } 2|p   
                         \end{array}\right. 
                          \end{align*}
                   \item If $c^t J^{-1} c \, \det J \, \equiv \, 3 \mbox { (mod } 4 \mbox {)}$, we then define $$\displaystyle{\chi_{c2}(p):=\left(\frac{-c^t J^{-1} c \, \det J }{p} \right )_K} \mbox {, i.e., a Dirichlet character modulo } 4|c^t J^{-1} c \, \det J|, \mbox { and } $$
            \begin{align*}
                \eta_{c2}(p):= \left\{ \begin{array}{cll} 
    1 & \mbox{if } p \, \equiv \, 1 \mbox { (mod } 4 \mbox {)}    \\ \\
                         (-1)^{(n+1)/2} & \mbox{if } p \, \equiv \, 3 \mbox { (mod } 4 \mbox {)} \quad \mbox {, i.e., a Dirichlet character modulo } 4. \\ \\
                         0 & \mbox{if } 2|p   
                         \end{array}\right.    \end{align*}         
                         \end{itemize}

Using the Dirichlet characters $\chi_{c1}, \chi_{c2}, \eta_{c1}$, and $\eta_{c2}$ defined above and observing that $(\frac{-1}{p})_K$ is the only non-trivial Dirichlet character modulo $4$, we can determine existence conditions of the pole $s=\frac{1}{2} -\frac{n}{2}$ in two cases as follow: 
\begin{itemize}
    \item If 
        $n \, \equiv \, 1 \mbox { (mod } 4 \mbox {)}$, then $c^t J^{-1} c \, \det J \not \equiv \, 3 \mbox { (mod } 4 \mbox {)}$ implies $\eta_{c1}=\chi_0$. Thus,  $\chi_{c1} \, \eta_{c1} = \chi_0$ if  $c^t J^{-1} c \, \det J$ is a perfect square. Otherwise,  $c^t J^{-1} c \, \det J\, \equiv \, 3 \mbox { (mod } 4 \mbox {)}$ yields $\eta_{c2} \, \neq \, \chi_0$. Hence, $\chi_{c2} \, \eta_{c2} = \chi_0$ if
        \begin{align*}
                \left(-\frac{c^t J^{-1} c \, \det J}{p}\right)_K= \eta_{c2}(p) = \left\{ \begin{array}{cll} 
    1 & \mbox{if } p \, \equiv \, 1 \mbox { (mod } 4 \mbox {)}    \\ \\
                         -1 & \mbox{if } p \, \equiv \, 3 \mbox { (mod } 4 \mbox {)} \quad  \\ \\
                         0 & \mbox{if } 2|p   
                         \end{array}\right.,
                          \end{align*}
                          which is not possible since $c^t J^{-1} c \, \det J \not\equiv \, 1 \mbox { (mod } 4 \mbox {)}$.
                             \item If 
     $n \, \equiv \, 3 \mbox { (mod } 4 \mbox {)}$, then $c^t J^{-1} c \, \det J \not\equiv \, 3 \mbox { (mod } 4 \mbox {)}$ gives us $\eta_{c1} \, \neq \, \chi_0$. Thus, $\chi_{c1} \, \eta_{c1} = \chi_0$ if
        \begin{align*}
                \left(\frac{c^t J^{-1} c \, \det J}{p}\right)_K= \eta_{c1}(p)= \left\{ \begin{array}{cll} 
    1 & \mbox{if } p \, \equiv \, 1 \mbox { (mod } 4 \mbox {)}    \\ \\
                         -1 & \mbox{if } p \, \equiv \, 3 \mbox { (mod } 4 \mbox {)} \quad  \\ \\
                         0 & \mbox{if } 2|p   
                         \end{array}\right., 
                          \end{align*}
                          which is impossible given that $c^t J^{-1} c \, \det J \not\equiv \, 3 \mbox { (mod } 4 \mbox {)}$.  Otherwise, $\eta_{c2}=\chi_0$ follows from $c^t J^{-1} c \, \det J \, \equiv \, 3 \mbox { (mod } 4 \mbox {)}$. This implies $\chi_{c2} \, \eta_{c2} = \chi_0$ if  $-c^t J^{-1} c \, \det J$ is a perfect square. 
    \end{itemize}
    Consequently, we have proved the lemma.
    \end{proof}
    \begin{rem}
    {(Existence of the pole at $s=0$)}
     \label{Remark}
    As we find a simple pole at $s=0$ in Lemma \ref{Lem 6.1}, we would expect to have the leading term of $N(B)$ of order $B^{n-1}$, which is different from the heuristically probabilistic expectation of order $B^{n-2}$. It actually turns out that the simple pole at $s=0$ contributes nothing to our asymptotic formula because the residue of $D(0,s)\,g(0,s)$ at this simple pole vanishes, as it also does vanish in the even degree case (see \cite{G}, Lemma 3.4). In other words, $D(0,s)\,g(0,s)$ is analytic at $s=0$. The key point of this vanishing is that at $s=0$, $\mbox{Res}_{\,s=0}(D(0,s)\,g(0,s))$ is solely controlled by $g(0,0)$, which turns out to be zero by observing
\begin{align*}
    \int_0^{\infty} \int_{\mathds{R}^n} w(z) \, \Phi \left ( \frac{F(z)}{y}, y \right )  \, dz \, \frac{dy}{y} = \int_0^{\infty} \int_{\mathds{R}^n} w(z)\, \Phi \left ( y,\frac{F(z)}{y} \right )  \, dz \, \frac{dy}{y}
\end{align*}
via Fubini's theorem and changing variables.
    \end{rem}

In view of the results presented in Sections \ref{Sec 4} and \ref{Sec 5} and Lemmas \ref{Lem 6.1}-\ref{Lem 6.2}-\ref{Lem 6.3}, $D(c,s)\,g(c,s)$, regarded as a function of $s$, admits the meromorphic continuation to $\sigma > -n/2$. For odd $n \geq 5$, it has three simple poles at $s=0, -1,$ and $1/2-n/2$ when $c=0$, and it has only one simple pole at $s=1/2-n/2$ when $c \neq 0$. With $n=3$ and $c=0$, the poles at $s=-1$ and $s=1/2-n/2$ are identical, which produces a double pole at $s=-1$ besides the simple pole at $s=0$ of $D(0,s)\,g(0,s)$. Invoking Cauchy's residue theorem, we prove our asymptotic formula in Theorem \ref{Thm 6.4}:
\begin{thm}  
\label{Thm 6.4} Let $\epsilon > 0$ and denote by $\square$ a perfect square. If $n \geq 5$ is odd,
\begin{align*}
    N(B)=c_1 B^{n-2} + c_2 B^{(n-1)/2}+O_{J,\epsilon, \omega}(B^{n/2+\epsilon-1}),
\end{align*}
where
\begin{align*}
    c_1&=\mbox{\normalfont{Res}}_{\,s=-1}\,(D(0,s)\,g(0,s)),\\
    \mbox{ and } 
    c_2&=
   \sum_{\substack{c \, \in \, \mathds{Z}^n\\ c^t J^{-1} c \,=\,0}} \mbox{\normalfont{Res}}_{\,s=1/2-n/2}\,D(c,s)+ \sum_{\substack{0 \, \neq \, \, c \, \in \, \mathds{Z}^n\\ c^t J^{-1} c \, \neq \, 0 \\ i^{(n-1)^2/2}\,c^t J^{-1} c\,\det J\,  =\, \square}} \mbox{\normalfont{Res}}_{\,s=1/2-n/2}\,(D(c,s)\,g(c,s)).
    \end{align*}
Morever, if $n=3$,
\begin{align*}
    N(B)=c_1' B \log B + c_2' B +O_{J,\epsilon, \omega}(B^{1/2+\epsilon}),
\end{align*}
where 
\begin{align*}
    c_1'&=\frac{-4}{\pi^4} \int_{\mathbb{R}} Z( z,1) \, dz,\\
    \mbox{ and } 
    c_2'&=
    \sum_{\substack{c \, \in \, \mathbb{Z}^n\\ c^t J^{-1} c\, =\,0}} \mbox{\normalfont{Res}}_{\,s=-1}\,(D(c,s)\,g(c,s))+\sum_{\substack{0 \, \neq \, \, c \, \in \, \mathbb{Z}^n\\ c^t J^{-1} c \, \neq \, 0 \\ -c^t J^{-1} c\,\det J\,  =\, \square}} \mbox{\normalfont{Res}}_{\,s=-1}\,(D(c,s)\,g(c,s)).
\end{align*}
\end{thm}
\begin{rem}
The Dirichlet series $D(c,s)$, for all $c \in \mathbb{Z}^n$, are given explicitly in Theorems \ref{Thm 4.1}-\ref{Thm 4.2}-\ref{Thm 4.3} in Section \ref{Sec 4}, and $Z( z,1)$ is defined as in (\ref{eq5.1}) of Proposition \ref{Prop 5.2} in Section \ref{Sec 5}.
\end{rem}
\begin{proof}
By Proposition \ref{Prop 2.2}, the zero-counting function $N(B)$ is of the form
\begin{align*}
\frac{c_B  B^{n-1}}{2 \pi i} \sum_{c \in \mathbb{Z}^n} \int_{\mathrm{Re}(s)=\sigma}D(c,s)B^{s}g(c,s)\,ds,
\end{align*} 
where $c_B=1+O_N(B^{-N})$ for any positive integer $N$.

For odd $n \geq 5$, all the poles of $D(c,s)\,g(c,s)$ collected in Lemmas \ref{Lem 6.1}, \ref{Lem 6.3}, and \ref{Lem 6.3} are simple. Since $g(c,s)$ is rapidly decreasing as $|t|$ goes to infinity given $\sigma_2 < \sigma < \sigma_1$ as commented in Section \ref{Sec 5}, we apply the Cauchy's residue theorem to shift our original line integral at $\sigma=\sigma_0 >1$ to the line integral at $\sigma=-n/2+\epsilon$ and obtain
\begin{align*}
    N(B)
    = &c_B  B^{n-1}  \, \mbox{Res}_{s \rightarrow 0} (D(0,s)\,g(0,s)) + c_B  B^{n-2} \, \mbox{Res}_{s \rightarrow -1} (D(0,s)\,g(0,s)) +\\&+ c_B  B^{(n-1)/2}   \, \mbox{Res}_{s \rightarrow 1/2-n/2} (D(0,s)\,g(0,s))+ \\&+ c_B  B^{(n-1)/2}  \sum_{\substack{0 \, \neq \, \, c \, \in \, \mathbb{Z}^n\\ c^t J^{-1} c \,=0}} \, \mbox{Res}_{\,s=1/2-n/2}\,(D(c,s)\,g(c,s))+ \tag{6.2} \label{eq6.2}
    \\&+ c_B  B^{(n-1)/2}  \sum_{\substack{0 \, \neq \, \, c \, \in \, \mathbb{Z}^n\\ c^t J^{-1} c \, \neq \, 0 \\ i^{(n-1)^2/2} \,c^t J^{-1} c\,\mbox{\tiny{det}}J\,  =\, \square}} \, \mbox{Res}_{\,s=1/2-n/2}\,(D(c,s)\,g(c,s))+ \tag{6.3} \label{eq6.3}
    \\&+ \frac{c_B  B^{n-1}}{2 \pi i}  \sum_{\substack{0 \, \neq \, \, c \, \in \, \mathbb{Z}^n\\ c^t J^{-1} c \, \neq \, 0 \\ i^{(n-1)^2/2} \, c^t J^{-1} c\,\mbox{\tiny{det}}J\,  \, \neq \,\, \square}} \, \int_{\mathrm{Re}(s)=-n/2+\sigma}D(c,s)B^{s}g(c,s)\,ds. \tag{6.4} \label{eq6.4}
\end{align*}

The first three terms in $N(B)$ accounts for the case $c=0$. Our previous comment, Remark \ref{Remark}, on the residue of $D(0,s)\,g(0,s)$ at $s=0$ make the first term in $N(B)$ zero. 

We wish to show (\ref{eq6.2}) and (\ref{eq6.3}) converges absolutely. Indeed, when $c\, \neq \, 0$ and $c^t J^{-1} c=0$, the factor $\zeta(2s+n)$ of $D(c,s)$ contributes the simple pole $s=1/2-n/2$ to $D(c,s)\,g(c,s)$, and we have
\begin{align*}
     \mbox{Res}_{\,s=1/2-n/2}\,(D(c,s)\,g(c,s))
    = \lim_{s \rightarrow 1/2-n/2} ((s-1/2+n/2)D(c,s)\,g(c,s)),
    \end{align*}
in which Theorems \ref{Thm 4.2} says $D(c,1/2-n/2)/\zeta(2s+n)$ is bounded by a power function of $c$, while Proposition \ref{Prop 5.1} asserts that $g(c,1/2-n/2)$ is rapidly decreasing in $c$. Thus, the infinite sum over $c$ (\ref{eq6.2}) is simply an absolutely convergent p-series. Likewise, we can apply these arguments to $L(s+n/2+1/2,\chi_0) $ of $D(c,s)$ in Theorem \ref{Thm 4.3} when $c\, \neq \, 0$ and $c^t J^{-1} c \, \neq \, 0$, and obtain absolutely convergent property for (\ref{eq6.3})as desired.

Moreover, we claim that the last term (\ref{eq6.4}) contributes to our error term of order $B^{n/2-1+\sigma}$. By Lebesgue's dominated convergence theorem, it suffices to show that
\begin{align*}
    \int_{-\infty}^{\infty} \sum_{\substack{0 \, \neq \, \, c \, \in \, \mathbb{Z}^n\\ c^t J^{-1} c \, \neq \, 0 \\ i^{(n-1)^2/2} \, c^t J^{-1} c\,\det J\,  \, \neq \,\, \square}} \, D(c,-n/2+\epsilon+it)\,g(c,-n/2+\epsilon+it)\,dt = O(1). \tag{6.5} \label{eq6.5}
\end{align*}
Indeed, we apply Proposition \ref{Prop 5.1} and Theorem \ref{Thm 4.3} again to see that the summand of the inner sum in (\ref{eq6.5}) is bounded by
\begin{align*}
    \frac{|c|^{3n/4-7/4+\epsilon}}{|c|^N } \frac{L(1/2+\epsilon,\kappa_{c,J})}{L(1+2\epsilon,\kappa_{c,J})}\frac{1}{\sqrt{\sigma_1^2+t^2+1}} \left ( \frac{1}{\sigma_2+n/2} - \frac{1}{\sigma_1-N} \right).
\end{align*}
Thus, the absolute convergence of the infinite sum over $c$, seen as a p-series, in (\ref{eq6.5}) tells us that the improper integral (\ref{eq6.5}) is in turn bounded by 
\begin{align*}
    \int_{-\infty}^{\infty} \frac{1}{\sqrt{\sigma_1^2+t^2+1}} \, dt 
    = O(1).
\end{align*}

For the boundary case $n=3$, the product rule of differentiation yields the residue of the double pole at $s=-1$, when $c=0$, with an extra term of $\ln B$, that is,
\begin{align*}
    \mbox{Res}_{\,s=-1}\,(D(0,s)B^s g(0,s))
    = &B^{-1} \ln B \lim_{s \rightarrow -1} ((s+1)^2D(0,s)\,g(0,s))+ \\&+ B^{-1} \mbox{Res}_{\,s=-1}\,(D(0,s)\,g(0,s)). \tag{6.6} \label{eq6.6}
\end{align*}
Furthermore, using (\ref{eq4.1}) in Theorem \ref{Thm 4.1} and Proposition \ref{Prop 5.2}, the limit in (\ref{eq6.6}) is equal to
\begin{align*}
     &\frac{4}{3}\frac{\zeta(0)\,\mbox{Res}_{\,s=-1} \zeta(2s+3)}{\zeta(2)}  \pi^{-3/2} \frac{\mbox{Res}_{\,s=-1} \Gamma((s+1)/2)}{\Gamma(1/2)} \int_{\mathbb{R}} Z( z,1) \, dz \\
     &= \frac{-4}{\pi^4} \int_{\mathbb{R}} Z( z,1) \, dz,
\end{align*}
where $Z( z,1)$ is defined as in (\ref{eq5.1}).
Hence, the asymptotic form for $n=3$ follows.

Consequently, we have proved the main theorem.
\end{proof}

\section{Acknowledgement}
The author gratefully thanks his advisor, Jayce Getz, for his consistent support on top of invaluable advice  and comments throughout the project. The author was supported by Getz's NSF DMS-1901883 for conference travel. The author also thanks L. Nguyen and J-M. Ho for their constant encouragements.
\section*{Appendix}
\label{appendix:Appendix}
As seen in Theorems \ref{Thm 4.1}-\ref{Thm 4.2}-\ref{Thm 4.3}, the expansion of $D(c,s)$ has a quotient of Riemann zeta functions or Dirichlet L-functions as one of its factors, we would not only expect to obtain its analytic continuation, but also its poles, particularly the pole at $s=1/2-n/2$, which contributes to the secondary term. Moreover, we discussed the possible poles of $D(c,s) \, g(c,s)$, which are at $s=0, -1$ and $s=1/2-n/2$, in Lemmas \ref{Lem 6.2}-\ref{Lem 6.2}-\ref{Lem 6.3}. For application, we wish to know whether or not these pole are cancelled by zeros of other factors of $D(c,s)$. For ease of exposition, we present our results here via Proposition \ref{Prop A1.1} for the case $c=0$, via Propositions \ref{Prop A2.1}-\ref{Prop A2.2} for the case $c \, \neq \, 0$ and $c^t J^{-1} c =0$, and finally via Propositions \ref{Prop A3.1}-\ref{Prop A3.2}-\ref{Prop A3.3} for the case $c \, \neq \, 0$ and $c^t J^{-1} c \neq 0$ below.

\begin{customprop}{A1.1}
\label{Prop A1.1}
\textit{
If $c = 0$, for $s=0, -1$ and $s=1/2-n/2$, one has
\begin{align*}
    \overline{D_2}(0,s):= D_2(0,s) \, \frac{(1-2^{-(s+1)})(1-2^{-(2s+n)})}{1-2^{-(2s+n+1)}} \, \neq \, 0. 
    \end{align*} }
 \end{customprop}
 \begin{proof} By the argument used in Theorem \ref{Thm 4.1}, $\overline{D_2}(0,s)$ is simplified to (\ref{eq4.1}), which is
\begin{align*}
    \frac{(1-2^{-(2s+n)}+d_{n,J}\,2^{-(2s+n+1)})}{1-2^{-(2s+n+1)}}.
\end{align*}
We claim that $\overline{D_2}(0,s) \, \neq \, 0$ at $s=0, -1$ and $s=1/2-n/2$. Indeed, $\overline{D_2}(0,s)$ vanishes at $s=0$ ($ s= \frac{1}{2} -\frac{n}{2}, \mbox { accordingly})$   if and only if $d_{n,J} = 2-2^{n+1}$ $(d_{n,J}=-2, \mbox { accordingly})$. We then recall $d_{n,J}$ (\ref{eq3.8}) from Corollary \ref{Cor 3.4.1}, 
\begin{align*}
                d_{n,J}= \left\{ \begin{array}{cll} 
    2^{(n-1)/2} & \mbox{if } 2m-n \, \equiv \, 1 \mbox { or } 7 \mbox { (mod } 8 \mbox {)}    \\ \\
                         -2^{(n-1)/2} & \mbox{if } 2m-n \, \equiv \, 3 \mbox { or } 5 \mbox { (mod } 8 \mbox {)} \quad  
                         \end{array}\right., 
                          \end{align*}  
to conclude that there is no such $n$ satisfying $d_{n,J} = 2-2^{n+1}$ or $d_{n,J}=-2$. Thus, $\overline{D_2}(0,s) \neq 0$ at  $s= 0$ and $s=1/2-n/2$.

Moreover, for $s=-1$, $\overline{D_2}(0,-1) =0$ yields $n=3$ and $4-n \, \equiv \, 5 \mbox { (mod } 8 \mbox {)}$, which is a contradiction.
Hence, we obtain the proposition as desired.
\end{proof}

\begin{customprop}{A2.1}
\label{Prop A2.1}
\textit{
If $c \, \neq \, 0$ and $c^t J^{-1} c =0$, one has
\begin{align*}
    D_2(c,s) \, \frac{1-2^{-(2s+n)}} {1-2^{-(2s+n+1)}} \, \neq \, 0 \mbox { at } s = 1/2-n/2. \end{align*} }
 \end{customprop}
 \begin{proof}
It suffices to check if $D_2(c, 1/2-n/2)$ vanishes. We recall $D_2(c,s)$ discussed in Lemma \ref{Lem 4.2.2} with two cases depending on the parity of $c$.
\begin{itemize}
    \item If $2 \nmid c$, $D_2(c,s)$ has the form of
 \begin{align*}
        D_2(c,s) = \frac{1-2^{-(2s+n)}+ d_{n,J} \, 2 ^{-(2s+n+1)}}{1-2^{-(2s+n)}},
    \end{align*}
which implies that it vanishes at $s=1/2-n/2$ if and only if $1-2^{-1}+d_{n,J}2^{-2}=0$, and so $d_{n,J}=-2$. This is impossible by checking the definition of $d_{n,J}$ in (\ref{eq3.8}).
        \item If $2 | c$, i.e., $2^\alpha \,||\, c$, for some positive integer $\alpha$,  $D_2(c,1/2-n/2)$ is equal to
    \begin{align*}
       \sum_{k=0}^{\alpha} 2^{k(\frac{n-3}{2})}+ \frac{d_{n, J} \, }{2^{n-1}-2} \left (\sum_{k=2}^{\alpha+2} 2^{k(\frac{n-3}{2})} - \sum_{l=1}^{\lfloor (\alpha+2)/2 \rfloor} \frac{1}{2^{l}} - \sum_{l=1}^{\lfloor (\alpha+1)/2 \rfloor} \frac{2^{(\frac{n-3}{2})}}{2^l} \right ). \tag{A1} \label{eqA1}
    \end{align*}
  In particular, if $n=3$, then $D_2(c,-1)=0$ means
  \begin{align*}
      \alpha +1 - ( \alpha +1 - (1- 2^{-\lfloor (\alpha+2)/2 \rfloor})-(1- 2^{-\lfloor (\alpha+1)/2 \rfloor}))=0,
      \end{align*}
      and so
      \begin{align*}
          2^{-\lfloor (\alpha+2)/2 \rfloor}+2^{-\lfloor (\alpha+1)/2 \rfloor}-2=0
      \end{align*}
   which contradicts to the assumption $\alpha \geq 1$. On the other hand, with $n \geq 5$, we use $d_{n,J}$ in (\ref{eq3.8}) and geometric sum to rewrite (\ref{eqA1}) as
  \begin{align*}
      \frac{1-2^{(\frac{n-3}{2})(\alpha+1)}}{1-2^{(\frac{n-3}{2})}}\pm \frac{2^{(\frac{n-1}{2})}}{2^{n-1}-2} \Big(  &2^{n-3} \frac{1-2^{(\frac{n-3}{2})(\alpha+1)}}{1-2^{(\frac{n-3}{2})}}  + \\
      &-  (1- 2^{-\lfloor \frac{\alpha+2}{2} \rfloor})  -2^{(\frac{n-3}{2})}(1- 2^{-\lfloor \frac{\alpha+2}{2} \rfloor})  \Big).
  \end{align*}
  We now consider the parity of $\alpha$ to get rid of the floor function. By checking $2$-power factors, we find that $D_2(c,1/2-n/2)=0$ implies $\alpha=n-5$ when $\alpha$ is even. Substituting $\alpha=n-5$ to $D_2(c,1/2-n/2)=0$ and simplifying, we get
  \begin{align*}
      &2(2^{3(\frac{n-3}{2})-1}-2^{(n-1)(\frac{n-3}{2})-1} -2^{(\frac{n-3}{2})-1} +1 - 2^{(n-1)(\frac{n-5}{4})}+2^{(\frac{n-1}{2})})
  \\
  &= 2^{n-2}(1-2^{(n-4)(\frac{n-3}{2})}-2^{(n-4)(\frac{n-3}{2})-n+2}). \tag{A2} \label{eqA2}
  \end{align*}
  If $n=5$, the equality (\ref{eqA2}) yields $-2=-10$, which is false. Otherwise, it implies $n=3$ and $n \geq 7$, which is also false. Likewise, we obtain $\alpha=n-4$ when $\alpha$ is odd. Upon substitution and factorization, we get
  \begin{align*}
      &2(2^{3(\frac{n-3}{2})-1}-2^{(n)(\frac{n-3}{2})-1}  +1 - 2^{n-4})
  \\
  &= 2^{n-2}(1-2^{(n-3)(\frac{n-3}{2})}-2^{(n-3)(\frac{n-3}{2})-n+2}). \tag{A3} \label{eqA3}
  \end{align*}
  Then the case $n=5$ gives us $-26=-28$, while the remaining case, $n \geq 7$, yields $n=3$. These false statements confirm no solution to (\ref{eqA3}).
    \end{itemize} 
  Hence, we have proved that the pole at $s=1/2-n/2$ does not vanish as desired.
\end{proof}

\begin{customprop}{{A2.2}}
\label{Prop A2.2}
\textit{
If $c \, \neq \, 0$ and $c^t J^{-1} c =0$, one has
\begin{align*}
    \prod_{ p | c, \, p \, \neq \, 2} D_p(c,s) \left [   \frac{1-p^{-(2s+n)}} {1-p^{-(2s+n+1)}}  \right] > 0 \mbox { at } s=1/2-n/2.
\end{align*} }
\end{customprop}
\begin{proof}
By using (\ref{eq4.2}) in Lemma 4.2.1, it suffices to show that $D_p(c,1/2-n/2) \, > \, 0$, where 
\begin{align*}
    &D_p(c,1/2-n/2)\\
    &= \sum_{k=0}^\alpha p ^{k(\frac{n-3}{2})} + \frac{p^{-(1-\alpha \frac{n-3}{2})}+p^{-(n-1-(\alpha+1)( \frac{n-3}{2}))}-p^{-(n-2+\lceil \frac{2+\alpha}{2} \rceil)}-p^{-(\frac{n-1}{2}+\lceil \frac{1+\alpha}{2} \rceil)}}{1-p^{2-n}}.
\end{align*}
We observe that, for $\alpha \geq 1$,
\begin{align*}
    n-2+\lceil \frac{2+\alpha}{2} \rceil &> n-1-(\alpha+1)( \frac{n-3}{2}), \\
 \mbox { and }   \frac{n-1}{2}+\lceil \frac{1+\alpha}{2} \rceil &> 1-\alpha \frac{n-3}{2}.
\end{align*}
Thus, by positivity, $D_p(c,1/2-n/2)$ is strictly positive.
\end{proof}

\begin{customprop}{A3.1}
\label{Prop A3.1}
\textit{
If $c \, \neq \, 0$ and $ c^t J^{-1} c \, \neq \, 0$, one has
\begin{align*}
    \prod_{p\nmid c, \, p| c^t J^{-1} c , \, p \, \neq \, 2} \frac{D_p(c,s)}{\rho_{c,J,s}(p)} > 0 \mbox { for } s =1/2 -n/2.
\end{align*} }
\end{customprop}
\begin{proof}
Using the expression of $D_p(c,s)$ in the proof of Lemma \ref{Lem 4.3.2} and evaluating it at $s=1/2-n/2$, we obtain
\begin{align*}
    D_p(c,1/2-n/2)
    &= 1+ \frac{ i^{(n-1)( p-1)^2/4}}{p^{\alpha/2+1}} \,\left(\frac{ \det J}{p} \right )_K \, \left(\frac{c^t J^{-1} c/p^\alpha}{p} \right )_K \,  \, \mathds{1} _{2|\alpha}+ \\
    &\quad + \frac{\, \mathds{1} _{2 \nmid \alpha}}{p^{\alpha/2+3/2}}+\sum_{k=\lceil 1+\alpha/2 \rceil }^\alpha \frac{ (1-p^{-1})}{p^{k(1/2)}} \,   \, \mathds{1} _{2 | k}. 
\end{align*}
By positivity, the pole $s=1/2-n/2$ does not vanish on $D_p(c,s)$. Indeed, \begin{align*}
     1- \frac{1}{p^{\alpha/2+1}} +\sum_{k=\lceil 1+\alpha/2 \rceil }^\alpha \frac{ (1-p^{-1})}{p^{k(1/2)}} \,   \, \mathds{1} _{2 | k} > 0.
\end{align*}
Since $\displaystyle{\rho_{c,J,1/2-n/2}(p)=1+ \frac{\chi_0 (p) }{p}}$, the proposition follows.
\end{proof}
\begin{customprop}{A3.2}
\label{Prop A3.2}
\textit{
If $c \, \neq \, 0$ and $ c^t J^{-1} c \, \neq \, 0$, one has
\begin{align*}
    \prod_{p|c, \, p \, \neq \, 2} \frac{D_p(c,s)}{\rho_{c,J,s}(p)} > 0 \mbox { for } s = 1/2-n/2,
\end{align*} }
\end{customprop}
\begin{proof}
Using the argument in Lemma \ref{Lem 4.3.1}, we get 
\begin{align*}
    D_p(c,1/2-n/2) = (X_3+X_4+X_5+X_6)|_{s=\frac{1-n}{2}},
\end{align*}
where
\begin{align*}
    X_3|_{s=\frac{1-n}{2}}
    &= \sum_{k=\alpha+\beta+1}^{2\alpha+\beta+1} \frac{p^{n(\alpha+\beta/2+1/2)-2\alpha-\beta-3/2}\,i^{(n-1)( p-1)^2/4} \,\left(\frac{ \det J}{p} \right )_K }{p^{k(n/2-1/2)}}  \left (\frac{c^t J^{-1} c /p^{2\alpha+\beta}}{p} \right )_K \, \mathds{1} _{2|\beta} ,\\
    X_4|_{s=\frac{1-n}{2}}&=\sum_{k=0}^\alpha p^{k\left(\frac{n-3}{2}\right)}, \quad 
    X_5|_{s=\frac{1-n}{2}}=\sum_{k=\alpha+\beta+1}^{2\alpha+\beta+1} \frac{-p^{n(1/2+\alpha+\beta/2)-2-2\alpha-\beta}  }{p^{k(n/2-1/2)}}   \, \mathds{1} _{2\nmid \beta}, \\ \mbox{ and } X_6|_{s=\frac{1-n}{2}}&= \sum_{k=\mbox{\tiny{max}}\{\lceil 1+\alpha+\beta/2 \rceil,\alpha+2\}}^{2\alpha+\beta} \frac{p^{ -1}(p-1)}{p^{k\left(\frac{3-n}{2}\right)}}   \sum_{l=\mbox{\tiny{max}}\{\lceil k-\alpha-\beta/2 \rceil,\lceil\frac{k-\alpha}{2}\rceil \}}^{\lfloor \frac{k}{2}\rfloor}  p^{(2-n)l}. 
\end{align*}
We observe that $X_6|_{s=\frac{1-n}{2}} > 0$. Moreover, $X_3|_{s=\frac{1-n}{2}}, X_4|_{s=\frac{1-n}{2}}$, and $X_5|_{s=\frac{1-n}{2}}$ have the same number of terms in their summation expansions.
If $2 | \beta$, $D_p(c,1/2-n/2)=0$ means the vanishing of
\begin{align*}
    &\sum_{k=\alpha+\beta+1}^{2\alpha+\beta+1} \frac{p^{n(\alpha+\beta/2+1/2)-2\alpha-\beta-3/2}\,i^{(n-1)( p-1)^2/4} \,\left(\frac{ \det J}{p} \right ) }{p^{k(n/2-1/2)}}  \left (\frac{c^t J^{-1} c /p^{2\alpha+\beta}}{p} \right )+\\&+\sum_{k=0}^\alpha p^{k\left(\frac{n-3}{2}\right)}+X_6|_{s=\frac{1-n}{2}},
\end{align*}
which only might be possible in the case of
\begin{align*}
    -\sum_{k=\alpha+\beta+1}^{2\alpha+\beta+1} \frac{p^{n(\alpha+\beta/2+1/2)-2\alpha-\beta-3/2} }{p^{k(n/2-1/2)}}  +\sum_{k=0}^\alpha p^{k\left(\frac{n-3}{2}\right)}+X_6|_{s=\frac{1-n}{2}}=0. \tag{A4} \label{eqA4}
\end{align*}
However, by our previous observations and noting that the leading power of the first summation in (\ref{eqA4}) is $  \frac{\alpha(n-3)}{2}-\frac{\beta}{2}-1$, which is smaller than that of, $ \frac{\alpha(n-3)}{2}$, in the second summation in (\ref{eqA4}), the equality (\ref{eqA4}) does not hold for any $\alpha, \beta$, and $n$. In fact, $D_p(c,1/2-n/2) > 0$ in this case. Using a similar argument, we also conclude that $D_p(c,1/2-n/2) > 0$ when $2 \nmid \beta$. Hence, we have proved the proposition.
\end{proof}

\begin{customprop}{A3.3}
\label{Prop A3.3} 
\textit{
If $c \, \neq \, 0$ and $ c^t J^{-1} c \, \neq \, 0$, 
\begin{align*}
     \frac{D_2(c,s)}{\rho_{c,J,s}(2)}  \, \neq \, 0 \mbox { for } s = 1/2-n/2,
\end{align*}
unless $n=5$, $m=4$, and $c^t J^{-1} c =1$, where $m = \#\{b_r: b_r=-1, r=1, \ldots, n\}$. }
\end{customprop} 
\begin{proof}
Following Theorem \ref{Thm 3.5} and the argument presented in the proof of Lemma \ref{Lem 4.3.3}, we get different explicit expressions of $D_2(c,s)$ depending on values of $c$. It has the following form when $2 \nmid c$ and $2 \nmid c^t J^{-1} c$, 
 \begin{align*} 
       &1-\frac{1}{2^{2s+1+n/2}}  \, \chi_4(c^t J^{-1} c) \,  \sin \left(\frac{(2m-n) \pi}{4}\right) \, \mathds{1} _{c^t J^{-1} c\,=\,\pm1} +\\
       &+ 2^{n/2-3} \left (\frac{2}{\det J} \right)_L \left(\chi_{8,2}\left(c^t J^{-1} c\right)\sin \left(\frac{(2m-n) \pi}{4}\right)+\chi_{8,3}\left(c^t J^{-1} c\right)\cos \left(\frac{(2m-n) \pi}{4}\right) \right).
       \end{align*}
We wish to know when $D_2(c,1/2-n/2)$ vanishes. Indeed, for $c^t J^{-1} c=1$, $D_2(c,1/2-n/2)=0$ if and only if $1-2^{(n-5)/2}=0$ and $\sin ((2m-n) /4)=2^{-1/2}$, which yield $n=5$ and $m=4$. Thus, the pole $s=1/2-n/2$ gets cancelled by a zero of $D_2(c,s)$ when $n=5$, $m=4$, and $c^t J^{-1} c =1$. For the remaining case when $c^t J^{-1} c \, \neq \, 1$, we get no solution for $D_2(c,1/2-n/2)=0$ after checking the condition of $m$, i.e., $2m > n$, and signs of $\sin x$ and $\cos x$. 

Now, when $2^\alpha \,||\, c^t J^{-1} c$, with $\alpha \in \mathbb{Z}_{\geq 1}$, we use geometric sum to simplify $D_2(c,1/2-n/2)$ to
\begin{align*}
    &1+2^{n/2-1} \cos \left(\frac{(2m-n) \pi}{4}\right) (1-(1/2)^{\lfloor \frac{\alpha}{2} \rfloor}) - 2^{\frac{n-\alpha-3}{2}} \cos \left(\frac{(2m-n) \pi}{4}\right) \, \mathds{1} _{2 | (\alpha +1)} +\\&-
    2^{\frac{n-\alpha-4}{2}} \sin \left(\frac{(2m-n) \pi}{4}\right) \chi_4(c^t J^{-1} c/2^\alpha) \, \mathds{1} _{2 | (\alpha +2)} +\\
    &+ 2^{\frac{n+\alpha-6}{2}} \left (\frac{2}{\det J} \right)_L \left(\chi_{8,2}\left(c^t J^{-1} c/2^\alpha\right)\sin \frac{(2m-n) \pi}{4}+\chi_{8,3}\left(c^t J^{-1} c/2^\alpha\right)\cos \frac{(2m-n) \pi}{4} \right). 
\end{align*}
We further consider parity of $\alpha$ to proceed. When $2 | \alpha$, upon simplification, finding the vanishing condition of $D_2(c,1/2-n/2)$ is amount to finding solutions of the following equations
\begin{align*}
    2^{\frac{\alpha}{2}} (1 \opm 2^{\frac{n-3}{2}})&=  2^{\frac{n-5}{2}} ( \opm\, 1 \opm 2) \tag{A5} \label{eqA5} \\
    \mbox { or } 2^{\frac{\alpha}{2}} (1 \opm 2^{\frac{n-3}{2}})&=  2^{\frac{n-5}{2}} ( \opm\, 1 \opm 2 \opm 2^\alpha), \tag{A6} \label{eqA6}
\end{align*}
where the notation $\opm$ can take either $+$ or $-$ sign, and it is chosen independently from others. 
By checking all the possible values of $\opm$, the equation (\ref{eqA5}) yields a possible solution with $\alpha=2$ and $n=7$, which satisfies $2^{\frac{\alpha}{2}} (1 - 2^{\frac{n-3}{2}})=  2^{\frac{n-5}{2}} ( - 1 - 2)$. Moreover, after considering signs of $\cos x, \sin x$, residues $x$ of $\chi_{8,2}(x), \chi_{8,3}(x), \chi_4(x)$, and noting that $- c^t J^{-1} c \det J  = \square$ and $y^2 \, \equiv \, 1 (\mbox{\small{mod }} 8) $, where $y$ is odd, the possible solution $\alpha=2$ and $n=7$ is not in the domain of the equation. Specifically, it yields $c^t J^{-1} c /4 = - y^2 \, \equiv \, 5 (\mbox{\small{mod }} 8)$, which is a contradiction to the odd parity of $y$. Likewise, the second equation (\ref{eqA6}) also yields $\alpha=2$ and $n=7$ as a candidate of its solutions, which results from $2^{\frac{\alpha}{2}} (1 + 2^{\frac{n-3}{2}})=  2^{\frac{n-5}{2}} ( - 1 + 2 + 2^\alpha)$. However, upon considering the domain conditions mentioned earlier, it ends up with 
 \begin{align*}
  \left\{ \begin{array}{cl} 
    c^t J^{-1} c /4 & \, \equiv \, 1 \,\mbox {(mod } 8  \mbox {)}    \\  c^t J^{-1} c /4 & \, \equiv \, 3 \,\mbox {(mod } 4  \mbox {)}                       \end{array}\right., 
                          \end{align*}
which is also a contradiction. Thus, $D_2(c,1/2-n/2)$ does not vanish in the even case of $\alpha$. We then consider the odd case of $\alpha$, which actually requires no effort to see that $D_2(c,1/2-n/2)$ has no solution by observing the fact that $2^\alpha \,||\, c^t J^{-1} c$ implies $|c^t J^{-1} c \det J   | \, \neq \, \square $, providing $\alpha$ is odd and $\det J = \pm 1$. Thus, we have verified the proposition for the case $2 \nmid c$. 

It remains to deal with the case $2|c$, i.e., $2^{\alpha} \,||\,c$ and $2^{2 \alpha +\beta} \,||\, c^t J^{-1} c$, for some integers $\alpha \geq 1, \beta \geq 0$. Using arguments in Lemma \ref{Lem 4.3.3}, we have 
\begin{align*}
           D_2(c,s) =
           \sum_{k=2}^{\alpha+2} \sum_{j=0}^{k-2} \frac{  2^{j(n/2-1)+k(n/2+1)-1} d_{n,j,k,J} \,  \mathds{1} _{2|k-j}}{2^{k(n+s+1)}} + \sum_{k=0}^{\alpha} \frac{1}{2^{k(s+1)}}, \tag{A7} \label{eqA7}
       \end{align*}
where  $d_{n,j,k,J}$ is given in (\ref{eq3.6}) as
 \begin{align*} d_{n,j,k,J}
     =\left\{ \begin{array}{cll} 
    0 & \mbox{if } 2^{k-j-2} \nmid c_{j,J}    \\ \\
                         - \, \chi_4(c_{j,J}/2^{k-j-2}) \,  2^{n/2} \sin \left(\frac{(2m-n) \pi}{4}\right) & \mbox{if } 2^{k-j-2} \,||\, c_{2,j,J}       \\ \\
                           - \,  2^{n/2} \cos \left(\frac{(2m-n) \pi}{4}\right) & \mbox{if } 2^{k-j-1} \,||\, c_{j,J}       \\ \\  2^{n/2} \cos \left(\frac{(2m-n) \pi}{4}\right) & \mbox{if } 2^{k-j} | c_{j,J}  \end{array}\right.  . 
    \end{align*}
We then expand the double summation in  (\ref{eqA7}), using notation $d_{n,j,k,J}^{(2)}:=2^{n/2} \cos \left(\frac{(2m-n) \pi}{4}\right)$ and $d_{n,j,k,J}^{(1)}:=- \, \chi_4(c_{2,j,J}/2^{k-j-2}) \,  2^{n/2} \sin \left(\frac{(2m-n) \pi}{4}\right)$, to get
\begin{align*}
     &-\frac{  2^{j(n/2-1)+k(n/2+1)-1} d_{n,j,k,J}^{(2)} \,  \mathds{1} _{2|k-j}}{2^{k(n+s+1)}} \, \mathds{1} _{\beta=1, j=\alpha, k=\alpha+2} + \tag{A8} \label{eqA8} \\
    &+ \frac{  2^{j(n/2-1)+k(n/2+1)-1} d_{n,j,k,J}^{(1)} \,  \mathds{1} _{2|k-j}}{2^{k(n+s+1)}} \, \mathds{1} _{\beta=0, j=\alpha, k=\alpha+2} + \tag{A9} \label{eqA9}\\
    &+ \sum_{k=2}^{\alpha+1} \sum_{j=0}^{k-2} \frac{  2^{j(n/2-1)+k(n/2+1)-1} d_{n,j,k,J}^{(2)} \,  \mathds{1} _{2|k-j}}{2^{k(n+s+1)}} \, \mathds{1} _{\beta=0, 1} + \tag{A10} \label{eqA10}\\
     &+  \sum_{j=0}^{\alpha-2} \frac{  2^{j(n/2-1)+k(n/2+1)-1} d_{n,j,k,J}^{(2)} \,  \mathds{1} _{2|k-j}}{2^{(\alpha+2)(n+s+1)}} \, \mathds{1} _{\beta=0, k=\alpha+2} + \tag{A11} \label{eqA11}\\
     &+  \sum_{j=0}^{\alpha-1} \frac{  2^{j(n/2-1)+k(n/2+1)-1} d_{n,j,k,J}^{(2)} \,  \mathds{1} _{2|k-j}}{2^{(\alpha+2)(n+s+1)}} \, \mathds{1} _{\beta=1, k=\alpha+2} + \tag{A12} \label{eqA12}\\
      &+ \sum_{k=2}^{\alpha+2} \sum_{j=0}^{k-2} \frac{  2^{j(n/2-1)+k(n/2+1)-1} d_{n,j,k,J}^{(2)} \,  \mathds{1} _{2|k-j}}{2^{k(n+s+1)}} \, \mathds{1} _{\beta \geq 2}. \tag{A13} \label{eqA13}
\end{align*}
Since (\ref{eqA8})-(\ref{eqA13}) are mutually exclusive, we have 6 subcases for $D_2(c, 1/2-n/2)$. In fact, (\ref{eqA10}) and (\ref{eqA13}) only differ by one more term in the summation over $k$.  That is, we can easily deduce $D_2(c, 1/2-n/2)$ for the case (\ref{eqA13}) from that of (\ref{eqA10}) by replacing $\alpha$ by $\alpha +1$. Indeed, with (\ref{eqA10}), $D_2(c,1/2-n/2)$ can be simplified to
\begin{align*}
    &\frac{\cos (\frac{(2m-n) \pi}{4})}{2^{n/2-1}-2^{1-n/2}} \left ( \sum_{k=2}^{\alpha+1} 2^{k(\frac{n-3}{2})} - \sum_{l=1}^{\lfloor \frac{\alpha+1}{2} \rfloor} 2^{l(-1)}- 2^{(\frac{n-3}{2})}\sum_{l=1}^{\lfloor \frac{\alpha}{2} \rfloor} 2^{l(-1)}\right)+\sum_{k=0}^{\alpha} 2^{k(\frac{n-3}{2})} \\
    &=\frac{\cos (\frac{(2m-n) \pi}{4})}{2^{n/2-1}-2^{1-n/2}} \left ( \left (1+\frac{2^{n/2-1}-2^{1-n/2}}{\cos \frac{(2m-n) \pi}{4}}\right) \sum_{k=2}^{\alpha+1} 2^{k(\frac{n-3}{2})} - \sum_{l=1}^{\lfloor \frac{\alpha+1}{2} \rfloor} 2^{l(-1)}- 2^{(\frac{n-3}{2})}\sum_{l=1}^{\lfloor \frac{\alpha}{2} \rfloor} 2^{l(-1)}\right).
\end{align*}
By observing that for all $n \geq 3$,
\begin{align*}
    \Big | \frac{2^{n/2-1}-2^{1-n/2}}{\cos \frac{(2m-n) \pi}{4}} \Big | &\geq 1,
    2^{(\frac{n-3}{2})} > 2^{-1},
   \mbox { and } \lfloor \frac{\alpha+1}{2} \rfloor + \lfloor   \frac{\alpha}{2} \rfloor = \alpha,
\end{align*}
 $D_2(c,1/2-n/2)$ can never be zero in this case (\ref{eqA10}), and so neither for the case (\ref{eqA13}). 
 
 With (\ref{eqA11}), upon simplification, $D_2(c,1/2-n/2)$ arrives at
 \begin{align*}
     2^{\alpha \frac{n-3}{2}+\frac{3n}{2}-4} \cos \left( \frac{(2m-n) \pi}{4} \right) \sum_{l=2}^{\lfloor \frac{\alpha+2}{2} \rfloor} 2^{l(2-n)} + \sum_{k=0}^{\alpha} 2^{k(\frac{n-3}{2})}.
 \end{align*}
For $n=3$, it comes down to 
\begin{align*}
    2^{1/2} \cos \left( \frac{(2m-3) \pi}{4} \right) \sum_{l=2}^{\lfloor \frac{\alpha+2}{2} \rfloor} 2^{l(-1)} + \alpha+1,
\end{align*}
which is evidently positive. In the case $n \geq 5$, by positivity, $D_2(c,1/2-n/2)=0$ means
\begin{align*}
     2^{\alpha \frac{n-3}{2}-\frac{1+n}{2}} (1-2^{(2-n)(\lfloor \frac{\alpha+2}{2} \rfloor-1)})(1-2^{\frac{n-3}{2}})&= (1-2^{(\alpha+1) \frac{n-3}{2}})(1-2^{2-n}),
 \end{align*}
and so
 \begin{align*}
         2^{\alpha \frac{n-3}{2}+\frac{n-5}{2}}(2^{(n-2)(\lfloor \frac{\alpha+2}{2} \rfloor-1)}-1)(1-2^{\frac{n-3}{2}})&=2^{(2-n)(\lfloor \frac{\alpha+2}{2} \rfloor-1)}(1-2^{(\alpha+1) \frac{n-3}{2}}) (2^{n-2}-1),
 \end{align*}
which implies 
\begin{align*}
  \alpha= \left\{ \begin{array}{cll} 
    \frac{5-n}{2n-5} & \mbox { if } 2 | \alpha    \\  \frac{3}{2n-5} & \mbox { if } 2 \nmid \alpha                        \end{array}\right.,   \end{align*}
which in turn contradicts to the condition $\alpha \geq 1$. Thus, $D_2(c,1/2-n/2)=0$ has no solution in the case (\ref{eqA11}) and subsequently in the case (\ref{eqA12}) by observing that $\, \mathds{1} _{2|k-j, k=\alpha+2}$ makes (\ref{eqA11}) and (\ref{eqA12}) differ by only $\beta$ in their computations toward $D_2(c, 1/2-n/2)$. 

It remains to check $D_2(c,1/2-n/2)$ in the cases (\ref{eqA8}) and (\ref{eqA9}). With (\ref{eqA9}), we get
\begin{align*}
    D_2(c,1/2-n/2)= - \, \chi_4(c_{2,j,J}) \,   \sin \left(\frac{(2m-n) \pi}{4}\right) \, 2^{\alpha \frac{n-3}{2}+\frac{n-4}{2}  } + \sum_{k=0}^{\alpha} 2^{k(\frac{n-3}{2})}.
\end{align*}
For $n=3$, $D_2(c,-1)$ vanishes if $-2^{-1 } + \alpha +1=0$, and so $\alpha=-1/2$, which contradicts to the positive condition of $\alpha$. On the other hand, the case $n \geq 5$ yields 
\begin{align*}
    -2^{\alpha \frac{n-3}{2}+\frac{n-5}{2}  } + \sum_{k=0}^{\alpha} 2^{k(\frac{n-3}{2})}=0, \tag{A14} \label{eqA14}
\end{align*}
which is false by considering parities of the two terms in the left-hand side of (\ref{eqA14}). In the same manner, we proceed the case (\ref{eqA8}) and get
\begin{align*}
    D_2(c,1/2-n/2)= -  \cos \left(\frac{(2m-n) \pi}{4}\right) \, 2^{\alpha \frac{n-3}{2}+\frac{n-4}{2}  } + \sum_{k=0}^{\alpha} 2^{k(\frac{n-3}{2})},
\end{align*}
which results to the same conclusion as we got in the case (\ref{eqA9}). 

Hence, $D_2(c,1/2-n/2)$ does not vanish in the case $2 | c$. Consequently, we have proved the proposition.
\end{proof}

\bibliography{ref}{}
\bibliographystyle{alpha}
\end{document}